\documentclass[a4paper,oneside,9pt]{article}%
\usepackage{makeidx}
\usepackage[english]{babel}
\usepackage{amsmath}
\usepackage{amsfonts}
\usepackage{amssymb}
\usepackage{stmaryrd}
\usepackage{graphicx}
\usepackage{mathrsfs}
\usepackage[colorlinks,linkcolor=red,anchorcolor=blue,citecolor=blue,urlcolor=blue]{hyperref}
\usepackage[symbol*,ragged]{footmisc}
\usepackage{bbm}
\usepackage{CJK}
\usepackage{indentfirst}
\usepackage{cases}

\usepackage{dsfont}
\usepackage{amsthm}

\allowdisplaybreaks[4]
\providecommand{\U}[1]{\protect\rule{.1in}{.1in}}
\providecommand{\U}[1]{\protect \rule{.1in}{.1in}}

\newtheorem{theorem}{Theorem}[section]

\newtheorem{corollary}[theorem]{Corollary}

\newtheorem{lemma}[theorem]{Lemma}

\newtheorem{proposition}[theorem]{Proposition}
\newtheorem{remark}[theorem]{Remark}

\numberwithin{equation}{section}

\begin{document}

\title{Bounds for moments of quadratic Dirichlet character sums of prime moduli }

\author{Yuetong Zhao
\vspace*{-4mm} \\
                    $\textrm{\small School of Mathematical Sciences, Beihang University}$
                     \vspace*{-4mm} \\
     \small  Beijing 100191, P. R. China
}

\footnotetext
   { \textit{ E-mail addresses}:
     \href{mailto:yuetong.zhao.math@gmail.com}{yuetong.zhao.math@gmail.com}
    }
    \date{}
\maketitle

{\textbf{Abstract}}: Assuming the generalized Riemann hypothesis, we evaluate sharp upper bounds for the shifted moments of quadratic Dirichlet $L$-functions with moduli 8$p$, where $p$ ranges over odd primes.  We then apply this result to prove bounds for the moments of quadratic Dirichlet character sums with prime moduli.

{\textbf{Mathematics Subject Classification (2010)}}: 11L40, 11M06

{\textbf{Keywords}}: quadratic Dirichlet characters; quadratic Dirichlet $L$-functions; prime moduli; shifted moments




\section{Introduction}
Let $\chi^{(d)}=\big(\frac{d}{\cdot}\big)$ be a real primitive Dirichlet character modulo $d$ given by Kronecker symbol.
In 1981, Jutila \cite{J3} investigated the first and second moment of the family of quadratic Dirichlet $L$-functions $L(1/2,\chi^{(d)})$
for $d$ running over fundamental discriminants and provided the asymptotic formulas for these moments. In 2000, applying random matrix theory, Keating and Snaith \cite{KS} conjectured that for any positive real number $k$,
\begin{equation*}
    \sum_{|d|\leq X}L(1/2,\chi^{(d)})^k\sim C_kX(\log X)^{\frac{k(k+1)}{2}},
\end{equation*}
where $C_k$ are explicit constants. In 2009, Soundararajan \cite{Sound} established upper bounds for moments of $L$-functions that are close to the predicted result under the generalized Riemann hypothesis (GRH). In 2013, Harper \cite{H} provided a refinement to obtain sharp upper bounds for the moments of $L$-functions conditionally. Instead of considering the values of $L$-functions at the central point, Szab\'o \cite{Szab} showed results for the shifted moments of $L$-functions which are very useful.

Also, in \cite{J3}, Jutila derived an asymptotic formula for the first moment of the family of quadratic Dirichlet $L$-functions $L(1/2,\chi^{(p)})$ with prime moduli, where the primes $p$ satisfy certain congruence conditions. After that, Baluyot and Pratt obtained an asymptotic formula for the second moment of $L(1/2,\chi^{(p)})$ under the GRH.
Recently, Gao and Zhao \cite{GZ} established the right order of magnitude for the k-th moment of $L(1/2,\chi^{(p)})$.
In addition, Andrade and Keating \cite{A} and Bui and Florea \cite{BF} studied the problem of the moments of quadratic Dirichle $L$-functions of prime moduli over function field.

The problems of quadratic Dirichlet character sums are one of the important applications of moments of $L$-functions. In 1973, M. Jutila \cite{J1} asserted a conjecture which stated that for any positive integer $m$, there exist constants $c_1(m)$, $c_2(m)$, with values depending only on $m$, such that
\begin{align}
\label{genJacobibound}
   \sum_{\substack {\chi \in \mathcal S(X) }} \big| \sum_{n \leq Y} \chi(n) \big|^{2m} \leq c_1(m)XY^m(\log X)^{c_2(m)},
\end{align}
  where $\mathcal S(X)$ denotes the set of all non-principal quadratic Dirichlet characters of modulus at most $X$.
  Also in the same year, in another paper by Jutila \cite{J2}, he established \eqref{genJacobibound} for $m=1$ with $c_2(m)=8$. This was later improved by Armon \cite{A} who showed that $c_2(m) = 1$ could be taken. We can find other related bounds in \cite{MV} and \cite{V}.

  In this year, Gao and Zhao \cite{GZ1} gave a smoothed version of the above conjecture of Jutila under GRH. After that, in \cite{GZ2}, they improved their result by using shifted moments of quadratic Dirichlet $L$-functions and gave an unsmoothed version. More precisely, they showed that for large $X$, $Y$ and any real number $m > 0$,
\begin{align}
\label{Smoothdef}
  \sum_{\substack{d \leq X \\ (d,2)=1,\,\mu^2(d)=1}}\Big | \sum_{n\leq Y}\chi^{(8d)}(n)\Big |^{2m} \ll XY^m(\log X)^{O_{m,\varepsilon}(1)},
\end{align}
where $\mu$ is M\"obius function.

In this paper, we obtain sharp upper bounds for character sums of the unsmoothed version, like (\ref{Smoothdef}), for prime moduli. First, we establish the upper bounds for the shifted moments of quadratic Dirichlet $L$-functions as follows.

\begin{theorem}
\label{t1}
 With the notation as above and the truth of GRH, let $k\geq 1$ be a fixed integer and $a_1,\ldots, a_{k}$, $A$ fixed positive real numbers. Suppose that $X$ is a large real number and $t=(t_1,\ldots ,t_{k})$ a real $k$-tuple with $|t_j|\leq  X^A$. Then
\begin{align*}
\begin{split}
 & \sum_{\substack{(p,2)=1 \\ p \leq X}}\log p\big| L\big(1/2+it_1,\chi^{(8p)} \big) \big|^{a_1} \cdots \big| L\big(1/2+it_{k},\chi^{(8p)}  \big) \big|^{a_{k}} \\
\ll & X(\log X)^{(a_1^2+\cdots +a_{k}^2)/4}  \prod_{1\leq m<\ell \leq k} \Big|\zeta \Big(1+i(t_m-t_{\ell})+\frac 1{\log X} \Big) \Big|^{a_ma_{\ell}/2}\\
& \times\Big|\zeta \Big(1+i(t_m+t_{\ell})+\frac 1{\log X} \Big) \Big|^{a_ma_{\ell}/2}\prod_{1\leq m\leq k} \Big|\zeta \Big(1+2it_m+\frac 1{\log X} \Big) \Big|^{a^2_m/4+a_m/2},
\end{split}
\end{align*}
 where $\zeta(s)$ is the Riemann zeta function. Here the implied constant depends on $k$, $A$ and the $a_j$'s, but not on $X$ or the $t_j$'s.
\end{theorem}

Suppose that $g:\mathbb{R}_{\geq 0} \rightarrow \mathbb{R}$ is defined by
\begin{equation} \label{gDef}
g(\alpha) =\begin{cases}
\log x,  & \text{if } \alpha\leq 1/\log x \text{ or } \alpha \geq e^x, \\
1/\alpha, & \text{if }   1/\log x \leq \alpha\leq 10, \\
\log \log \alpha, & \text{if }  10 \leq \alpha \leq e^{x}.
\end{cases}
\end{equation}

By applying Lemma \ref{cos} from Section 2 to Theorem \ref{t1}, we obtain a result related to the correlation factor $g(\alpha)$ as follows.
\begin{corollary}\label{cor1}
 With the notation as above and the truth of GRH, let $k\geq 1$ be a fixed integer and $a_1,\ldots, a_{k}$, $A$ fixed positive real numbers. Suppose that $X$ is a large real number and $t=(t_1,\ldots ,t_{k})$ a real $k$-tuple with $|t_j|\leq  X^A$. Then
\begin{align*}
  &\sum_{\substack{(p,2)=1 \\ p \leq X}}\log p\big| L\big(1/2+it_1,\chi^{(8p)} \big) \big|^{a_1} \cdots \big| L\big(1/2+it_{k},\chi^{(8p)}  \big) \big|^{a_{k}} \\
\ll &  X(\log X)^{(a_1^2+\cdots +a_{k}^2)/4} \prod_{1\leq m<\ell\leq k} g(|t_m-t_{\ell}|)^{a_ma_{\ell}/2}g(|t_m+t_{\ell}|)^{a_ma_{\ell}/2}\prod_{1\leq m\leq k} g(|2t_m|)^{a^2_m/4+a_m/2}.
\end{align*}
\end{corollary}

From this Corollary, we obtain the key result of this paper.
\begin{theorem}\label{t2}
With the notation as above and the truth of GRH, for any integer $k \geq 1$ and any real number $m$ satisfying $2m \geq k+1$, we have for large $X$, $Y$  and any $\varepsilon>0$,
\begin{align}
\label{mainestimate}
 S_m(X,Y):=\sum_{\substack{p \leq X \\ (p,2)=1}}\log p\Big | \sum_{n \leq Y}\chi^{(8p)}(n)\Big |^{2m} \ll XY^m(\log X)^{R(m,k,\varepsilon)},
\end{align}
where
\begin{align} \label{Edef}
 R(m,k,\varepsilon)
 = \max \Bigg(2m^2-m+1+\varepsilon,\,\, \frac{(2m-k)^2}{4}+&\frac{k(k-1)}{2}+2m+1+\varepsilon,\nonumber\\& m(2m+1)-k\Big(2m+\frac{3}{4}-k\Big)+\varepsilon\Bigg) .
 \end{align}
 In particular, by setting $k=2$, we obtain, for $m \geq 2$,
\begin{align}
\label{mainestime2}
 S_m(X,Y) \ll XY^m(\log X)^{2m^2-m+1+\varepsilon}.
\end{align}
\end{theorem}

\begin{remark}
From Theorem \ref{t2}, we derive that for any real number $m\geq 1$ and $\varepsilon>0$,
\begin{equation}\label{good}
    S_m(X,Y)\ll XY^m(\log X)^{O_{m,\varepsilon}(1)}.
\end{equation}
As for $0<m<1$, we can use H\"older's inequality to get
\begin{align*}
 S_m(X,Y) &\ll  X^{1-1/n}(S_{mn}(X,Y))^{1/n}\ll X^{1-1/n}(XY^{mn}(\log X)^{O_{m,\varepsilon}(1)})^{1/n}\\
 &\ll XY^m(\log X)^{O_{m,\varepsilon}(1)},
\end{align*}
where real number $n>1$ is large enough such that $mn\geq 1$. Then we can see that for any real number $m>0$, (\ref{good}) is valid.
\end{remark}

\section{Preliminaries}
\textbf{Notation.}
Let $p$ and $q$ always denote prime numbers. The symbol $X$ always denotes a large real number. The symbol $\varepsilon$ always denotes an arbitrary small positive number, which may not be the same in different occurrences. For any real number $t$, $\lfloor t\rfloor$ denotes its integer part. The symbol $\square$ denotes a square number. We use $\chi^{(c)}(n)$ to denote the Kronecker symbol $\big(\frac{c}{n}\big)$. We write $d|n$ to mean that $d$ divides $n$, and that $p^k\parallel n$ to mean that $p^k$ is the exact power of $p$ dividing $n$. The function $\Omega(n)$ denotes the number of prime factors of $n$, counting according to multiplicity. 

We give two Mertens' estimates and a result concerning a sum over primes.
\begin{lemma}\label{M}
    For $z\geq 2$ we have
 \begin{align*}
    &(a) \,\,\sum_{p\leq z}\frac{1}{p}=\log\log z+c+O(1/\log z), \quad \text{where c is a constant};\\
    &(b) \,\,\sum_{p\leq z}\frac{\log p}{p}=\log z+O(1).
 \end{align*}
\end{lemma}
\begin{proof}
    See Theorem 3.4 of Koukoulopoulos \cite{K1}.
\end{proof}

\begin{lemma}\label{cos}
Let $x\geq 2$ and $\alpha\geq 0$. Then, we have
\begin{align}
\sum_{p\leq x}\frac{cos(\alpha\log p)}{p}&=\log|\zeta(1+1/\log x+i\alpha)|+O(1)\label{2.1}\\
&\leq \log g(\alpha) +O(1),\label{2.2}
\end{align}
where the function $g(\alpha)$ is defined in (\ref{gDef}).
\end{lemma}
\begin{proof}
The equality (\ref{2.1}) is a special case of Lemma 3.2 of Koukoulopoulos \cite{K}.  The inequality (\ref{2.2}) can be found in Szab\'o \cite{Szab}.
\end{proof}

\begin{lemma}\label{L(s,chi)}
Under GRH, for every primitive Dirichlet character $\chi$ modulo $q$, with $\Re(s) \geq 1/2$ and for any $\varepsilon>0$, we have
\begin{align*}
 L(s, \chi) \ll |qs|^{\varepsilon}.
\end{align*}
\end{lemma}
\begin{proof}
   See Corollary 5.20 of Iwaniec and Kowalski \cite{IK}.
\end{proof}

We give the Stirling's approximation for the factorial function.
\begin{lemma}\label{stirling}
For $n\in\mathbb{N}$, we have
    \begin{equation*}
    n!=\left(\frac{n}{e}\right)^n\sqrt{2\pi n}(1+O(1/n)).
\end{equation*}
\end{lemma}
\begin{proof}
    See Theorem 1.12 of Koukoulopoulos \cite{K1}.
\end{proof}

\begin{lemma}\label{Heath}
Let $X$, $Z$ be positive integers, and $a_n$ be any complex numbers. Then
\begin{align*}
  &  \sum_{\substack{d \leq X \\ (d,2)=1}}\bigg|\sum_{n\leq Z} a_n\chi^{(8d)}(n)\bigg|^{2} \ll  (XZ)^{\varepsilon}(X+Z)\sum_{\substack{m,n \leq Z \\ mn=\square}}|a_ma_n|.
\end{align*}
\end{lemma}
\begin{proof}
See Corollary 2 of Heath-Brown \cite{HB}.
\end{proof}

\begin{lemma}\label{Y>X}
Let $X$, $Y$ be positive integers. For any $m>0$, we have
\begin{align*}
  &  \sum_{\substack{|D| \leq X, D\neq\square \\ 4|D \,\text{or}\,D\equiv1(\bmod 4)}}\bigg|\sum_{n\leq Y} \chi^{(D)}(n)\bigg|^{2m} \ll_m  X^{m+1}.
\end{align*}
\end{lemma}
\begin{proof}
See Theorem 1 of Armon \cite{A}.
\end{proof}

The following is Perron's formula.
\begin{lemma}\label{perron}
    Assume that the series
\begin{equation*}
    f(s)=\sum_{n=1}^{\infty}\frac{w_n}{n^s}, \quad s=\sigma+it,
\end{equation*}
    converges absolutely for $\sigma>1$, that $|w_n|\leq W(n)$, where $W(n)>0$ is a monotonically increasing function, and that
    \begin{equation*}
        \sum_{n=1}^{\infty}\frac{|w_n|}{n^{\sigma}}=O((\sigma-1)^{-\alpha}),\quad \alpha>0,
    \end{equation*}
    as $\sigma\rightarrow 1^+$. Then for any $1<b\leq b_0$, $T\geq 1$ and $x=N+1/2$, we have
    \begin{equation}
        \sum_{n\leq x}w_n=\frac{1}{2\pi i}\int_{b-iT}^{b+iT}f(s)\frac{x^s}{s}ds+O\Big(\frac{x^b}{T(b-1)^{\alpha}}\Big)+O\Big(\frac{xW(2x)\log x}{T}\Big),
    \end{equation}
    where the constants in the O symbols depend only on $b_0$.
\end{lemma}
\begin{proof}
    See Theorem 1 of Chapter 5 of Karatsuba \cite{Ka}.
\end{proof}

Let $\Phi$ be a non-negative smooth function
\begin{equation}\label{Phi}
\Phi(x)\begin{cases}
            \leq 1,\quad  x\in [1/4,1/2]\cup [1,3/2],\\
            =1,\quad  x\in[1/2,1],\\
            =0,\quad \text{otherwise }.
            \end{cases}
\end{equation}
The Mellin transform of $\Phi(x)$ is denoted by $\widehat{\Phi}(s)$, and  for any complex number $s$, it is given by the following integral:
\begin{equation}\label{Mellin}
    \widehat{\Phi}(s)=\int_0^{\infty}\Phi(x)x^{s-1}dx.
\end{equation}

\begin{lemma}\label{asymp}
   Assume $GRH$. Let $c$ be a positive odd integer, $\Phi(x)$ be a smooth function as described above, and let $k\in\mathbb{R}$ with $k\geq 0$. Then, for any $\varepsilon>0$, we have
      \begin{equation}\label{formula}
       \sum_{(p,2)=1}(\log p)\Big(1-\frac{1}{2p}\Big)^{-k}\chi^{(8p)}(c)\Phi\left(\frac{p}{X}\right)=\delta_{c=\square}\widehat{\Phi}(1)X+O(X^{1/2+\varepsilon}\log\log(c+2)).
   \end{equation}
\end{lemma}
\begin{proof}
    If $k=0$, see Lemma 2.4 in Gao and Zhao \cite{GZ}. If $k>0$, we have
    \begin{equation*}
        \sum_{(p,2)=1}(\log p)\Big(1-\frac{1}{2p}\Big)^{-k}\chi^{(8p)}(c)\Phi\left(\frac{p}{X}\right)=\sum_{(p,2)=1}(\log p)\chi^{(8p)}(c)\Phi\left(\frac{p}{X}\right)+O\left(\sum_{(p,2)=1}\frac{\log p}{p}\Phi\left(\frac{p}{X}\right)\right).
    \end{equation*}
    For the main term of the above expression, it is the same as the case $k=0$. Using Lemma \ref{M} $(b)$ to estimate the error term, we get a contribution of $\ll \log X$, which can be absorbed by the error term in (\ref{formula}).
\end{proof}

\begin{lemma}
\label{lemmasum}
  Let $k$ be a positive integer and let $Q,a_1,a_2,\ldots, a_{k}$ be fixed positive real constants, $x\geq 2$.  Define $a:=a_1+\cdots+ a_{k}$.  Assume $X$ is a large real number, $p$ is any odd prime satisfying $p \leq X$, $\sigma \geq 1/2$ and $t_1,\ldots, t_{k}$ are fixed real numbers with $|t_i|\leq X^Q$. For any integer $n$, define
$$h(n):=\frac{1}{2}\Re \Big( \sum^{k}_{m=1}a_mn^{-it_m} \Big).$$
Then, under the GRH, we have the following inequality holds:
\begin{align}\label{lem2.4}
 & \sum^{k}_{m=1}  a_m\log |(1-1/2p)L(\sigma+it_m,\chi^{(8p)})| \nonumber\\
    \leq & 2 \sum_{q\leq x} \frac{h(q)\chi^{(8p)}(q)}{q^{1/2+\max(\sigma-1/2, 1/\log x)}}\frac{\log x/q}{\log x}+
    \sum_{q\leq x^{1/2}} \frac{h(q^2)}{q^{1+2\max(\sigma-1/2, 1/\log x)}}+(Q+1)a\frac{\log X}{\log x}+O(1).
\end{align}
\end{lemma}

\begin{proof}
    By Lemma 2.7 of Gao and Zhao \cite{GZ2}, we have
\begin{align}\label{GZ}
& \sum^{k}_{m=1}  a_m\log |L(\sigma+it_m,\chi^{(8p)})| \nonumber\\
    \leq & 2 \sum_{q\leq x} \frac{h(q)\chi^{(8p)}(q)}{q^{1/2+\max(\sigma-1/2, 1/\log x)}}\frac{\log x/q}{\log x}
    +\sum_{q\leq x^{1/2}} \frac{h(q^2)\chi^{(8p)}(q^2)}{q^{1+2\max(\sigma-1/2, 1/\log x)}}+(Q+1)a\frac{\log X}{\log x}+O(1).
\end{align}
For the second term of the right hand side of the above inequality, we can get
\begin{align}
\label{sump}
 \sum_{q\leq x^{1/2}} \frac{h(q^2)\chi^{(8p)}(q^2)}{q^{1+2\max(\sigma-1/2, 1/\log x)}}=& \sum_{q\leq x^{1/2}} \frac{h(q^2)}{q^{1+2\max(\sigma-1/2, 1/\log x)}}
 -\frac{h(p^2)}{p^{1+2\max(\sigma-1/2, 1/\log x)}} \nonumber\\
 \leq & \sum_{q\leq x^{1/2}} \frac{h(q^2)}{q^{1+2\max(\sigma-1/2, 1/\log x)}}+ \frac{a}{2p}\nonumber\\
 \leq & \sum_{q\leq x^{1/2}} \frac{h(q^2)}{q^{1+2\max(\sigma-1/2, 1/\log x)}}- a\log \Big(1-\frac{1}{2p}\Big).
\end{align}
Here the final step uses $x\leq-\log(1-x)$ for any $0<x<1$. Inserting  (\ref{sump}) into (\ref{GZ}), we obtain the result.
\end{proof}

\begin{lemma}\label{rude}
Assume the truth of GRH, let $k\geq 1$ be a fixed integer, and ${\bf a}=(a_1,\ldots ,a_{k})$ and $\ t=(t_1,\ldots ,t_{k})$
 be real $k$-tuples with $a_i \geq 0$ for all $i$.  Define $a:=a_1+\cdots+ a_{k}$.  Then for sufficiently large real number $X$ and $\sigma \geq 1/2$, we have
\begin{align*}
\sum_{\substack{(p,2)=1 \\ p \leq X}}\log p\big| L\big(\sigma+it_1,\chi^{(8p)} \big) \big|^{a_1} \cdots \big| L\big(\sigma+it_{k},\chi^{(8p)}  \big) \big|^{a_{k}} \ll_{{\bf a}} &  X(\log X)^{a(a+1)/2}.
\end{align*}
\end{lemma}
\begin{proof}
    We can deduce this result by modifying the proof of Theorem 1.2 of Gao and Zhao \cite{GZ}.
\end{proof}

Let $M$ and $N$ be large numbers that depends only on $\bf a$.
Define
\begin{align}\label{notations}
\begin{split}
 &\beta_{0} = 0, \;\;\;\;\; \beta_{j} = \frac{20^{j-1}}{(\log\log X)^{2}}, \;\;\; s_j=2\lfloor e^N\beta_j^{-3/4}\rfloor, \;\;\;\;\; h_{j} = \frac{s_j}{100}, \;\;\; \mbox{for} \; j \geq 1, \\
&\mathcal{J} = 1 + \max\{j : \beta_{j} \leq 10^{-M} \}.
\end{split}
\end{align}
From the definition as above, we have
\begin{equation}\label{sizeJ}
    \beta_{\mathcal{J}-1}\leq 10^{-M}\Rightarrow 20^{\mathcal{J}-2}\leq10^{-M}(\log\log X)^2\Rightarrow \mathcal{J}\ll_{\bf a} \log \log\log X.
\end{equation}
Let $\gamma(n)$ denotes the multiplicative function such that $\gamma(p^\alpha)=\alpha !$. We see that
\begin{equation}\label{gamma}
    \gamma(n^2)\geq \gamma (n).
\end{equation}
Recall that the definition of $h(n)$ in Lemma \ref{lemmasum}, for any $1\leq i\leq j\leq \mathcal{J}$, let

\begin{equation}\label{aG}
    a{\mathcal G}_{i,j}(p)=\sum_{X^{\beta_{i-1}}<q\leq X^{\beta_i}}\frac{2h(q)\chi^{(8p)}(q)}{q^{1/2+\max(\sigma-1/2,1/\log X^{\beta_j})}}\frac{\log(X^{\beta_j}/q)}{\log X^{\beta_j}}
\end{equation}
and
\begin{equation}\label{t(n,x)}
t\big(n,X^{\beta_j}\big)=\prod_{q^{\alpha}||n}\left(\frac{2h(q)\log(X^{\beta_j}/q)}{aq^{\max(\sigma-1/2,1/\log X^{\beta_j})}\log X^{\beta_j}}\right)^{\alpha}.
\end{equation}
We can observe that $t(n,x)$ is a totally multiplicative function of $n$ and
\begin{equation}\label{sizet}
|t(n,x)|\leq 1.
\end{equation}

\section{Proof of Theorem 1.1}

Dividing $0 <p \leq X$ into dyadic blocks, we see that to prove Theorem \ref{t1}, it suffices to demonstrate the result

  \begin{align}
\label{Lprodboundssmoothed}
\begin{split}
  \sum_{(p,2)=1 } & \log p \big| L\big(\sigma+it_1,\chi^{(8p)} \big) \big|^{a_1} \cdots \big| L\big(\sigma+it_{k},\chi^{(8p)}  \big) \big|^{a_{k}}\Phi \Big( \frac p{X} \Big) \\
\ll & X(\log X)^{(a_1^2+\cdots +a_{k}^2)/4} \prod_{1\leq m<\ell \leq k} \big|\zeta(1+i(t_m-t_{\ell})+\frac 1{\log X}) \big|^{a_ma_{\ell}/2}\\
&\times \big|\zeta(1+i(t_m+t_{\ell})+\frac 1{\log X}) \big|^{a_ma_{\ell}/2}\prod_{1\leq m\leq k} \big|\zeta(1+2it_m+\frac 1{\log X}) \big|^{a^2_m/4+a_m/2}
\end{split}
\end{align}
for $\sigma=1/2$. To obtain a more general theorem, we will prove the result for $\sigma\geq 1/2$.

In (\ref{lem2.4}), we take $x=X^{\beta_j}$ and cut up our main Dirichlet polynomial into smaller pieces $a{\mathcal G}_{i,j}(p)\,(1\leq i\leq j\leq \mathcal{J})$.
In terms of the size of the short Dirichlet polynomial, we define the following sets:
\begin{align}\label{sets}
  \mathcal{S}(0) =& \{ (p,2)=1 : |a{\mathcal G}_{1,l}(p)| > h_{1} \; \text{ for some } 1 \leq l \leq \mathcal{J} \} ,   \nonumber\\
 \mathcal{S}(j) =& \{ (p,2)=1  : |a{\mathcal G}_{m,l}(p)| \leq
  h_{m},  \; \mbox{for all} \; 1 \leq m \leq j \; \mbox{and} \; m \leq l \leq \mathcal{J}, \nonumber\\
 & \;\;\;\;\;\;\;\;\;\;\; \text{but }  |a{\mathcal G}_{j+1,l}(p)| > h_{j+1} \; \text{ for some } j+1 \leq l \leq \mathcal{J} \} ,  \quad\quad  1\leq j \leq \mathcal{J}, \\
 \mathcal{S}(\mathcal{J}) =& \{(p,2)=1  : |a{\mathcal G}_{m,
\mathcal{J}}(p)| \leq h_{m}, \; \forall 1 \leq m \leq \mathcal{J}\}.\nonumber
\end{align}

Obviously, we see that
\begin{equation*}
    \{p:(p,2)=1\}=\bigcup_{j=0}^{\mathcal{J}}\mathcal{S}(j).
\end{equation*}

Next, we will estimate the sum (\ref{Lprodboundssmoothed})
over the sets $\mathcal{S}(0)$ and $\mathcal{S}(j)\,(1\leq j\leq \mathcal{J})$, respectively.

\subsection{The estimate of the sum over the set $\mathcal{S}(0)$}
First, we state the following lemma:
\begin{lemma}\label{sumPhi}
Let $\Phi$ and $\mathcal{S}(0)$ be defined as in (\ref{Phi}) and (\ref{sets}). Then, we have
\begin{equation*}
\sum_{\substack{(p,2)=1 \\p \in S(0) }} (\log p)\Phi \Big( \frac p{X} \Big)\ll_{\bf a} Xe^{-(\log\log X)^2}.
\end{equation*}
\end{lemma}
\begin{proof}
    By the definition of $\mathcal{S}(0)$ and Rankin's trick, we get
\begin{equation*}
    \sum_{\substack{(p,2)=1 \\p \in S(0) }} (\log p)\Phi \Big( \frac p{X} \Big)\leq \sum_{(p,2)=1}\sum_{l=1}^{\mathcal{J}}\big(h_1^{-1}|a\mathcal{G}_{1,l}(p)|\big)^{R}(\log p)\Phi \Big( \frac p{X} \Big),
\end{equation*}
where we choose
\begin{equation*}
    R=2\left\lfloor \frac{1}{100\beta_1}\right\rfloor.
\end{equation*}
By (\ref{notations}), (\ref{aG}) and (\ref{t(n,x)}), we obtain
\begin{align*}
&\sum_{(p,2)=1}\sum_{l=1}^{\mathcal{J}}\big(h_1^{-1}|a\mathcal{G}_{1,l}(p)|\big)^{2\left\lfloor \frac{1}{100\beta_1}\right\rfloor}(\log p)\Phi \Big( \frac p{X} \Big)\\
   = &  \sum_{l=1}^{\mathcal{J}} \left(\frac{50a}{\lfloor e^N\beta_1^{-3/4}\rfloor}\right)^{2\left\lfloor \frac{1}{100\beta_1}\right\rfloor}\sum_{\substack{n\\ \Omega(n)=2\left\lfloor \frac{1}{100\beta_1}\right\rfloor\\ q|n\Rightarrow X^{\beta_0}<q\leq  X^{\beta_1}}}\frac{\Big(2\left\lfloor \frac{1}{100\beta_1}\right\rfloor\Big)!\,\,t\big(n,X^{\beta_l}\big)}{\gamma(n) n^{1/2}}\sum_{(p,2)=1}(\log p)\chi^{(8p)}(n)\Phi \Big( \frac p{X} \Big).
\end{align*}
We apply Lemma \ref{asymp} to estimate the innermost sum over $p$ mentioned above. By (\ref{notations}), (\ref{sizeJ}) and (\ref{sizet}), the contribution from the error term is
\begin{align}\label{r1}
    &\ll X^{1/2+\varepsilon}\mathcal{J} \left(\frac{50a}{\lfloor e^N\beta_1^{-3/4}\rfloor}\right)^{2\left\lfloor \frac{1}{100\beta_1}\right\rfloor}\max_{1\leq l\leq \mathcal{J}}\sum_{\substack{n\\ \Omega(n)=2\left\lfloor \frac{1}{100\beta_1}\right\rfloor\\ q|n\Rightarrow X^{\beta_0}<q\leq  X^{\beta_1}}}\frac{\left(2\left\lfloor \frac{1}{100\beta_1}\right\rfloor\right)!\,\,t\big(n,X^{\beta_l}\big)\log \log (n+2)}{\gamma(n) n^{1/2}}\nonumber\\
    &\ll X^{1/2+\varepsilon}\mathcal{J}\left(\frac{50a}{\lfloor e^N\beta_1^{-3/4}\rfloor}\right)^{2\left\lfloor \frac{1}{100\beta_1}\right\rfloor}\sum_{\substack{n\\ \Omega(n)=2\left\lfloor \frac{1}{100\beta_1}\right\rfloor\\ q|n\Rightarrow X^{\beta_0}<q\leq  X^{\beta_1}}}\frac{\left(2\left\lfloor \frac{1}{100\beta_1}\right\rfloor\right)!}{\gamma(n) n^{1/2-\varepsilon}}\nonumber\\
    &=X^{1/2+\varepsilon}\mathcal{J}\left(\frac{50a}{\lfloor e^N\beta_1^{-3/4}\rfloor}\right)^{2\left\lfloor \frac{1}{100\beta_1}\right\rfloor}\left(\sum_{ X^{\beta_0}<q\leq  X^{\beta_1}}\frac{1}{q^{1/2-\varepsilon}}\right)^{2\left\lfloor \frac{1}{100\beta_1}\right\rfloor}\nonumber\\
    &\ll X^{1/2+1/100+\varepsilon}\mathcal{J}\left(\frac{50a}{\lfloor e^N\beta_1^{-3/4}\rfloor}\right)^{2\left\lfloor \frac{1}{100\beta_1}\right\rfloor}\ll_{\bf a} Xe^{-(\log\log X)^2}.
\end{align}
The contribution from the main term is
\begin{equation}\label{mainterm}
    \ll X\mathcal{J}\left(\frac{50a}{\lfloor e^N\beta_1^{-3/4}\rfloor}\right)^{2\left\lfloor \frac{1}{100\beta_1}\right\rfloor}\max_{1\leq l\leq \mathcal{J}}\sum_{\substack{n=\square\\ \Omega(n)=2\left\lfloor \frac{1}{100\beta_1}\right\rfloor\\ q|n\Rightarrow X^{\beta_0}<q\leq  X^{\beta_1}}}\frac{\left(2\left\lfloor \frac{1}{100\beta_1}\right\rfloor\right)!\,\,t\big(n,X^{\beta_l}\big)}{\gamma(n) n^{1/2}}.
\end{equation}
Since function $t\big(n,X^{\beta_l}\big)$ is totally multiplicative of $n$, and by (\ref{gamma}) and (\ref{sizet}), we get that (\ref{mainterm}) is
\begin{align}\label{mainterm2}
    &= X\mathcal{J}\left(\frac{50a}{\lfloor e^N\beta_1^{-3/4}\rfloor}\right)^{2\left\lfloor \frac{1}{100\beta_1}\right\rfloor}\max_{1\leq l\leq \mathcal{J}}\sum_{\substack{\Omega(n)=\left\lfloor \frac{1}{100\beta_1}\right\rfloor\\ q|n\Rightarrow X^{\beta_0}<q\leq  X^{\beta_1}}}\frac{\left(2\left\lfloor \frac{1}{100\beta_1}\right\rfloor\right)!\,\,t^2\big(n,X^{\beta_l}\big)}{\gamma(n^2) n}\nonumber\\
    &\ll X\mathcal{J}\left(\frac{50a}{\lfloor e^N\beta_1^{-3/4}\rfloor}\right)^{2\left\lfloor \frac{1}{100\beta_1}\right\rfloor}\sum_{\substack{\Omega(n)=\left\lfloor \frac{1}{100\beta_1}\right\rfloor\\ q|n\Rightarrow X^{\beta_0}<q\leq  X^{\beta_1}}}\frac{\left(2\left\lfloor \frac{1}{100\beta_1}\right\rfloor\right)!}{\gamma(n) n}\nonumber\\
    &=X\mathcal{J}\left(\frac{50a}{\lfloor e^N\beta_1^{-3/4}\rfloor}\right)^{2\left\lfloor \frac{1}{100\beta_1}\right\rfloor}\frac{\left(2\left\lfloor \frac{1}{100\beta_1}\right\rfloor\right)!}{(\left\lfloor \frac{1}{100\beta_1}\right\rfloor)!}\left(\sum_{ X^{\beta_0}<q\leq  X^{\beta_1}}\frac{1}{q}\right)^{\left\lfloor \frac{1}{100\beta_1}\right\rfloor}.
\end{align}
By Lemma \ref{M} $(a)$, Lemma \ref{stirling}, (\ref{notations}) and (\ref{sizeJ}), we can calculate (\ref{mainterm2}) is
\begin{align}\label{r2}
    &\ll X\mathcal{J}e^{-\left\lfloor \frac{1}{100\beta_1}\right\rfloor}\left(\frac{\log\log X}{25}\left(\frac{50a}{e^N}\right)^2\beta_1^{1/2}\right)^{\left\lfloor \frac{1}{100\beta_1}\right\rfloor}
    \ll X \mathcal{J} \Big(e\cdot\frac{e^{2N}}{100a^2}\Big)^{-\left\lfloor \frac{(\log\log X)^2}{100}\right\rfloor}\nonumber\\
    &\ll_{\bf a} Xe^{-(\log\log X)^2},
\end{align}
where we choose $N$ to be large enough. According to (\ref{r1}) and (\ref{r2}), We complete the proof of this lemma.
\end{proof}

Using Cauchy-Schwarz inequality, we have
\begin{align*}
\begin{split}
  & \sum_{\substack{(p,2)=1 \\p \in S(0) }}  \log p\big| L\big(\sigma+it_1,\chi^{(8p)} \big) \big|^{a_1} \cdots \big| L\big(\sigma+it_{k},\chi^{(8p)}  \big) \big|^{a_{k}} \Phi \Big( \frac p{X} \Big) \\
\leq &  \Big ( \sum_{\substack{(p,2)=1 \\p \in S(0) }} (\log p)\Phi \Big( \frac p{X} \Big)  \Big )^{1/2} \Big (
 \sum_{\substack{(p,2)=1 }}\log p\big| L\big(\sigma+it_1,\chi^{(8p)} \big) \big|^{2a_1} \cdots \big| L\big(\sigma+it_{k},\chi^{(8p)}  \big) \big|^{2a_{k}} \Phi \Big( \frac p{X} \Big) \Big)^{1/2}.
\end{split}
\end{align*}
Then, by Lemma \ref{sumPhi} and Lemma \ref{rude}, the above summation is $\ll_{\bf a} X$. Naturally, we have
\begin{align*}
& \sum_{\substack{(p,2)=1 \\p \in S(0) }}  \log p\big| L\big(\sigma+it_1,\chi^{(8p)} \big) \big|^{a_1} \cdots \big| L\big(\sigma+it_{k},\chi^{(8p)}  \big) \big|^{a_{k}} \Phi \Big( \frac p{X} \Big)\\
   \ll_{\bf a} & X(\log X)^{(a_1^2+\cdots +a_{k}^2)/4} \prod_{1\leq m<\ell \leq k} \big|\zeta(1+i(t_m-t_{\ell})+\frac 1{\log X}) \big|^{a_ma_{\ell}/2}\\
   &\times \big|\zeta(1+i(t_m+t_{\ell})+\frac 1{\log X}) \big|^{a_ma_{\ell}/2}\prod_{1\leq m\leq k} \big|\zeta(1+2it_m+\frac 1{\log X}) \big|^{a^2_m/4+a_m/2}.
\end{align*}

\subsection{The estimate of the sum over the set $\mathcal{S}(j)$}

 Fix a $j$ with $1 \leq j \leq \mathcal{J}$. From Lemma \ref{lemmasum}, we have
\begin{align*}
 &\sum_{\substack{(p,2)=1 \\p \in \mathcal{S}(j)}}\log p\big| L  \big(\sigma+it_1,  \chi^{(8p)} \big) \big|^{a_1} \cdots \big| L\big(\sigma+it_k,\chi^{(8p)} \big) \big|^{a_k}\Phi\Big(\frac{p}{X}\Big) \nonumber\\
\ll & \sum_{\substack{(p,2)=1 \\p \in \mathcal{S}(j)}}\exp \Big(
 a\sum^j_{i=1}{\mathcal G}_{i,j}(p)\Big )\exp \left(\sum_{p\leq X^{\beta_j/2}} \frac{h(q^2)}{q^{1+2\max(\sigma-1/2,1/\log X^{\beta_j})}} \right) \nonumber\\
 &\quad\quad\quad \times\exp \left(\frac {(Q+1)a}{\beta_j} \right)(\log p)\left(1-\frac{1}{2p}\right)^{-a}\Phi \Big(\frac{p}{X}\Big).
\end{align*}
From the above inequality, the definition of $\mathcal{S}(j)$, and Rankin's trick, we obtain
\begin{align}
\label{Sj1}
 &\sum_{\substack{(p,2)=1 \\p \in \mathcal{S}(j) }}\log p\big| L  \big(\sigma+it_1,  \chi^{(8p)} \big) \big|^{a_1} \cdots \big| L\big(\sigma+it_{k},\chi^{(8p)}  \big) \big|^{a_{k}}\Phi\Big(\frac {p}{X}\Big)\nonumber \\
\ll & \sum_{\substack{(p,2)=1 \\p \in \mathcal{S}(j) }}\sum_{t=j+1}^{\mathcal{J}}(h_{j+1}^{-1}|a\mathcal{G}_{j+1,t}(p)|)^{2\big\lfloor\frac{1}{100\beta_{j+1}}\big\rfloor} \exp \Big (a\sum^j_{i=1}{\mathcal G}_{i,j}(p)\Big )\exp \left(\sum_{p\leq X^{\beta_j/2}} \frac{h(q^2)}{q^{1+2\max(\sigma-1/2,1/\log X^{\beta_j})}} \right)\nonumber\\
 & \quad \quad \times\exp \left(\frac {(Q+1)a}{\beta_j} \right)(\log p)\left(1-\frac{1}{2p}\right)^{-a}\Phi \Big( \frac p{X} \Big).
\end{align}
Using Taylor's formula with integral remainder and Stirling's formula (Lemma \ref{stirling}), we can deduce that for any $z \in \mathbb{C}$,
\begin{align}\label{ezrelation}
  e^z=\sum^{n-1}_{s=0}\frac {z^s}{s!}+O\Big( \frac {|z|^n}{n!}e^{|z|}e^z \Big).
 \end{align}
 Let $z=a {\mathcal G}_{i,j}(p)$, $n=s_i+1$ in (\ref{ezrelation}). Since $|a {\mathcal G}_{i,j}(p)| \leq h_i$ for $p \in \mathcal{S}(j)$, we get
\begin{equation}\label{exp}
    \exp \Big ( a {\mathcal G}_{i,j}(p) \Big )  =\sum^{s_i}_{s=0}\frac {( a{\mathcal G}_{i,j}(p)) ^s}{s!} \left( 1+   O(e^{-s_i}) \right).
\end{equation}
Inserting (\ref{exp}) into (\ref{Sj1}), we obtain
 \begin{align}\label{Sj2}
 &\sum_{\substack{(p,2)=1 \\p \in \mathcal{S}(j) }}\log p\big| L  \big(\sigma+it_1,  \chi^{(8p)} \big) \big|^{a_1} \cdots \big| L\big(\sigma+it_{k},\chi^{(8p)}  \big) \big|^{a_{k}}\Phi\Big(\frac {p}{X}\Big) \nonumber\\
 \ll & \sum_{\substack{(p,2)=1 \\p \in \mathcal{S}(j) }}\sum_{t=j+1}^{\mathcal{J}}(h_{j+1}^{-1}|a\mathcal{G}_{j+1,t}(p)|)^{2\big\lfloor\frac{1}{100\beta_{j+1}}\big\rfloor} \prod_{i=1}^{j}\left(\sum^{s_i}_{s=0}\frac {( a{\mathcal G}_{i,j}(p)) ^s}{s!}\right)\nonumber\\
 &\times\exp \left(\sum_{p\leq X^{\beta_j/2}} \frac{h(q^2)}{q^{1+2\max(\sigma-1/2,1/\log X^{\beta_j})}} \right) \exp \left(\frac {(Q+1)a}{\beta_j} \right)(\log p)\left(1-\frac{1}{2p}\right)^{-a}\Phi \Big( \frac p{X} \Big).
 \end{align}
By (\ref{notations}), (\ref{aG}) and (\ref{t(n,x)}), we calculate that the Dirichlet polynomial
\begin{align}\label{poly1}
    &(h_{j+1}^{-1}|a\mathcal{G}_{j+1,t}(p)|)^{2\big\lfloor\frac{1}{100\beta_{j+1}}\big\rfloor} \nonumber\\
    =&\left(\frac{50a}{\big\lfloor e^N\beta^{-3/4}_{j+1}\big\rfloor}\right)^{2\big\lfloor\frac{1}{100\beta_{j+1}}\big\rfloor}\sum_{\substack{n\\ \Omega(n)=2\big\lfloor\frac{1}{100\beta_{j+1}}\big\rfloor\\ q|n\Rightarrow X^{\beta_{j}}<q\leq X^{\beta_{j+1}}}}\frac{(2\big\lfloor\frac{1}{100\beta_{j+1}}\big\rfloor)!\,\,t\big(n,X^{\beta_t}\big)}{\gamma(n) n^{1/2}}
    \chi^{(8p)}(n)
\end{align}
and
 \begin{align}\label{poly2}
    \prod_{i=1}^{j}\left(\sum^{s_i}_{s=0}\frac {( a{\mathcal G}_{i,j}(p)) ^s}{s!}\right)
    &=\prod_{i=1}^{j}\sum_{s=0}^{s_i} \left(\frac{a^s}{s!}\right)\left(\sum_{X^{\beta_{i-1}}<q\leq X^{\beta_i}}\frac{t\big(q,X^{\beta_j}\big)\chi^{(8p)}(q)}{q^{1/2}}\right)^{s}\nonumber\\
    &=\prod_{i=1}^{j}\sum_{\substack{n_i\\ \Omega(n_i)\leq s_i\\ q|n_i\Rightarrow X^{\beta_{i-1}}<q\leq X^{\beta_i}}}\frac{a^{\Omega(n_i)}\,\,t\big(n_i,X^{\beta_j}\big)}{\gamma(n_i) n_i^{1/2}} \chi^{(8p)}(n_i).
\end{align}
We observe that $(n,n_i)=1$ for any $1\leq i\leq j$. By inserting these two formulas (\ref{poly1}) and (\ref{poly2}), back into (\ref{Sj2}), we get
\begin{align}\label{Sj3}
 &\sum_{\substack{(p,2)=1 \\p \in \mathcal{S}(j) }}\log p\big| L  \big(\sigma+it_1,  \chi^{(8p)} \big) \big|^{a_1} \cdots \big| L\big(\sigma+it_{k},\chi^{(8p)}  \big) \big|^{a_{k}}\Phi\Big(\frac {p}{X}\Big) \nonumber\\
 \ll & \sum_{t=j+1}^{\mathcal{J}}\left(\frac{50a}{\big\lfloor e^N\beta^{-3/4}_{j+1}\big\rfloor}\right)^{2\big\lfloor\frac{1}{100\beta_{j+1}}\big\rfloor}
 \sum_{\substack{n\\ \Omega(n)=2\big\lfloor\frac{1}{100\beta_{j+1}}\big\rfloor\\ q|n\Rightarrow X^{\beta_j}<q\leq X^{\beta_{j+1}}}}\frac{(2\big\lfloor\frac{1}{100\beta_{j+1}}\big\rfloor)!\,\,t\big(n,X^{\beta_t}\big)}{\gamma(n) n^{1/2}}\nonumber\\
   & \times\prod_{i=1}^{j}\left(\sum_{\substack{n_i\\ \Omega(n_i)\leq s_i\\ q|n_i\Rightarrow X^{\beta_{i-1}}<q\leq X^{\beta_i}}}\frac{a^{\Omega(n_i)}\,\,t\big(n_i,X^{\beta_j}\big)}{\gamma(n_i) n_i^{1/2}} \right)
   \exp \left(\sum_{q\leq X^{\beta_j/2}} \frac{h(q^2)}{q^{1+2\max(\sigma-1/2,1/\log X^{\beta_j})}} \right) \nonumber\\
   &\times\exp \left(\frac {(Q+1)a}{\beta_j} \right)
 \sum_{(p,2)=1 }(\log p)\left(1-\frac{1}{2p}\right)^{-a}\chi^{(8p)}(nn_i)\Phi \Big( \frac p{X} \Big).
 \end{align}
Applying Lemma \ref{asymp} to estimate the innermost sum of (\ref{Sj3}). By (\ref{sizeJ}) and (\ref{sizet}), the contribution of the error term is
\begin{align}\label{error1}
\ll& X^{1/2+\varepsilon}\left(\frac{50a}{\big\lfloor e^N\beta^{-3/4}_{j+1}\big\rfloor}\right)^{2\big\lfloor\frac{1}{100\beta_{j+1}}\big\rfloor}\sum_{n\leq X^{1/50}}\frac{(2\big\lfloor\frac{1}{100\beta_{j+1}}\big\rfloor)!}{n^{1/2}}\sum_{n'\leq X^{\sum_{i=1}^{\mathcal{J}}\beta_i s_i}}\frac{a^{\sum_{i=1}^{\mathcal{J}} s_i}}{n'^{1/2}}\nonumber\\
&\times\exp \left(\sum_{q\leq X^{\beta_j/2}} \frac{h(q^2)}{q^{1+2\max(\sigma-1/2,1/\log X^{\beta_j})}} \right) \exp \left(\frac {(Q+1)a}{\beta_j} \right).
\end{align}
From (\ref{notations}), we estimate that
\begin{equation*}
    \sum_{1\leq i\leq \mathcal{J}}\beta_i s_i\leq 50e^N\cdot 10^{-M/4},\quad \sum_{1\leq i\leq \mathcal{J}}s_i\leq 2e^N(\log\log X)^{3/2}.
\end{equation*}
Hence, in (\ref{error1}), the sum over $n'$ is also a short Dirichlet polynomial. Then, the error term is
\begin{align}\label{error2}
    & \ll X^{1/2+\varepsilon} \left(a^{2e^N(\log\log X)^{3/2}}\left(\frac{50a}{\big\lfloor e^N\beta^{-3/4}_{j+1}\big\rfloor}\right)^{2\big\lfloor\frac{1}{100\beta_{j+1}}\big\rfloor}\Big(2\big\lfloor\frac{1}{100\beta_{j+1}}\big\rfloor\Big)!\right)\nonumber\\
    &\quad\times\left(\sum_{n\leq X^{1/20}}\frac{1}{n^{1/2}}\right)\exp \left(\sum_{q\leq X^{\beta_j/2}} \frac{h(q^2)}{q^{1+2\max(\sigma-1/2,1/\log X^{\beta_j})}} \right) \exp \left(\frac {(Q+1)a}{\beta_j} \right).
\end{align}
By Lemma \ref{stirling} and (\ref{notations}), we can calculate that
\begin{equation*}
    a^{2e^N(\log\log X)^{3/2}}\left(\frac{50a}{\big\lfloor e^N\beta^{-3/4}_{j+1}\big\rfloor}\right)^{2\big\lfloor\frac{1}{100\beta_{j+1}}\big\rfloor}\Big(2\big\lfloor\frac{1}{100\beta_{j+1}}\big\rfloor\Big)!\ll_{\bf a} X^{\varepsilon}.
\end{equation*}
Using this estimation, (\ref{error2}) is
\begin{align}\label{error3}
 & \ll_{\bf a} X^{1/2+1/40+\varepsilon}\exp \left(\sum_{q\leq X^{\beta_j/2}} \frac{h(q^2)}{q^{1+2\max(\sigma-1/2,1/\log X^{\beta_j})}} \right) \exp \left(\frac {(Q+1)a}{\beta_j} \right).
 \end{align}
The contribution of the main term of the innermost sum in (\ref{Sj3}) is
\begin{align}\label{Main}
&\ll X\sum_{t=j+1}^{\mathcal{J}}\left(\frac{50a}{\big\lfloor e^N\beta^{-3/4}_{j+1}\big\rfloor}\right)^{2\big\lfloor\frac{1}{100\beta_{j+1}}\big\rfloor}\left(\sum_{\substack{n=\square\\ \Omega(n)=2\big\lfloor\frac{1}{100\beta_{j+1}}\big\rfloor\\ q|n\Rightarrow X^{\beta_j}<q\leq X^{\beta_{j+1}}}}\frac{(2\big\lfloor\frac{1}{100\beta_{j+1}}\big\rfloor)!\,\,t\big(n,X^{\beta_t}\big)}{\gamma(n) n^{1/2}}\right)\nonumber\\
 &\times\prod_{i=1}^{j}\left(\sum_{\substack{n_i=\square\\ \Omega(n_i)\leq s_i\\ q|n_i\Rightarrow X^{\beta_{i-1}}<q\leq X^{\beta_i}}}\frac{a^{\Omega(n_i)}\,\,t\big(n_i,X^{\beta_j}\big)}{\gamma(n_i) n_i^{1/2}} \right)\exp \left(\sum_{q\leq X^{\beta_j/2}} \frac{h(q^2)}{q^{1+2\max(\sigma-1/2,1/\log X^{\beta_j})}} \right) \exp \left(\frac {(Q+1)a}{\beta_j} \right).
\end{align}
Observe that
\begin{align}\label{sum1}
&\sum_{\substack{n_i=\square\\ \Omega(n_i)\leq s_i\\ q|n_i\Rightarrow X^{\beta_{i-1}}<q\leq X^{\beta_i}}}\frac{a^{\Omega(n_i)}\,\,t\big(n_i,X^{\beta_j}\big)}{\gamma(n_i) n_i^{1/2}}\nonumber\\
=&\sum_{\substack{n_i=\square\\q|n_i\Rightarrow X^{\beta_{i-1}}<q\leq X^{\beta_i}}}\frac{a^{\Omega(n_i)}\,\,t\big(n_i,X^{\beta_j}\big)}{\gamma(n_i) n_i^{1/2}}-\sum_{\substack{n_i=\square\\ \Omega(n_i)> s_i\\ q|n_i\Rightarrow X^{\beta_{i-1}}<q\leq X^{\beta_i}}}\frac{a^{\Omega(n_i)}\,\,t\big(n_i,X^{\beta_j}\big)}{\gamma(n_i) n_i^{1/2}}.
\end{align}
For the first sum of (\ref{sum1}), we can write the summation as products over primes
\begin{align}\label{sum11}
    \sum_{\substack{n_i=\square\\q|n_i\Rightarrow X^{\beta_{i-1}}<q\leq X^{\beta_i}}}\frac{a^{\Omega(n_i)}\,\,t\big(n_i,X^{\beta_j}\big)}{\gamma(n_i) n_i^{1/2}}&=\sum_{\substack{n_i\\q|n_i\Rightarrow X^{\beta_{i-1}}<q\leq X^{\beta_i}}}\frac{a^{2\Omega(n_i)}\,\,t^2\big(n_i,X^{\beta_j}\big)}{\gamma(n_i^2) n_i}\nonumber\\
    &=\prod_{X^{\beta_{i-1}}<q\leq X^{\beta_i}}\left(1+\frac{a^2t^2(q,X^{\beta_j})}{2q}+O\Big(\frac{1}{q^2}\Big)\right)\nonumber\\
    &\leq \exp\left(\sum_{X^{\beta_{i-1}}<q\leq X^{\beta_i}}\left(\frac{a^2t^2(q,X^{\beta_j})}{2q}+O\Big(\frac{1}{q^2}\Big)\right)\right).
\end{align}
The last expression above follows from the well-known inequality $1+x\leq e^x$ for any $x\in\mathbb{R}$. By (\ref{t(n,x)}), we can write
\begin{equation*}
    t(q, X^{\beta_j})=\frac{2h(q)}{aq^{\max(\sigma-1/2,1/\log X^{\beta_j})}}\Big(1+O\Big(\frac{\log q}{\log X^{\beta_j}}\Big)\Big),\quad q\leq X^{\beta_j}.
\end{equation*}
Hence, the expression (\ref{sum11}) is
\begin{equation}\label{sum12}
    \leq \exp\left(\sum_{X^{\beta_{i-1}}<q\leq X^{\beta_i}}\frac{2h^2(q)}{q^{1+2\max(\sigma-1/2,1/\log X^{\beta_{j}})}}+O\Bigg(\sum_{X^{\beta_{i-1}}<q\leq X^{\beta_i}}\Big(\frac{\log q}{q\log X^{\beta_j}}+\frac{1}{q^2}\Big)\Bigg)\right).
\end{equation}
By Rankin's trick and the similar arguments of getting (\ref{sum12}), for the second term of (\ref{sum1}), we have
\begin{align}\label{sum22}
    &\sum_{\substack{n_i=\square\\ \Omega(n_i)> s_i\\ q|n_i\Rightarrow X^{\beta_{i-1}}<q\leq X^{\beta_i}}}\frac{a^{\Omega(n_i)}\,\,t\big(n_i,X^{\beta_j}\big)}{\gamma(n_i) n_i^{1/2}}
    =\sum_{\substack{n_i\\ \Omega(n_i)> s_i/2\\ q|n_i\Rightarrow X^{\beta_{i-1}}<q\leq X^{\beta_i}}}\frac{a^{2\Omega(n_i)}\,\,t^2\big(n_i,X^{\beta_j}\big)}{\gamma(n_i^2) n_i}\nonumber\\
    &\leq \sum_{\substack{n_i\\  q|n_i\Rightarrow X^{\beta_{i-1}}<q\leq X^{\beta_i}}}\frac{e^{\Omega(n_i)-s_i/2}a^{2\Omega(n_i)}\,\,t^2\big(n_i,X^{\beta_j}\big) }{\gamma(n_i^2) n_i}\nonumber\\
    &=e^{-s_i/2}\prod_{X^{\beta_{i-1}}<q\leq X^{\beta_i}}\left(1+\frac{ea^2t^2(q,X^{\beta_j})}{2q}+O\Big(\frac{1}{q^2}\Big)\right)\nonumber\\
    &\ll e^{-s_i/2}\exp\left(\sum_{X^{\beta_{i-1}}<q\leq X^{\beta_i}}\frac{2h^2(q)}{q^{1+2\max(\sigma-1/2,1/\log X^{\beta_{j}})}}+O\Bigg(\sum_{X^{\beta_{i-1}}<q\leq X^{\beta_i}}\Big(\frac{\log q}{q\log X^{\beta_j}}+\frac{1}{q^2}\Big)\Bigg)\right).
\end{align}
From (\ref{sum1}), (\ref{sum12}) and (\ref{sum22}), we have
\begin{align}\label{sum10}
     &\prod_{i=1}^{j}\left(\sum_{\substack{n_i=\square\\ \Omega(n_i)\leq s_i\\ q|n_i\Rightarrow X^{\beta_{i-1}}<q\leq X^{\beta_i}}}\frac{a^{\Omega(n_i)}\,\,t\big(n_i,X^{\beta_j}\big)}{\gamma(n_i) n_i^{1/2}} \right)\nonumber\\
     \leq & \prod_{i=1}^{j}(1+O(e^{-s_i/2}))\nonumber\\
     &\times\exp\left(\sum_{X^{\beta_{i-1}}<q\leq X^{\beta_i}}\frac{2h^2(q)}{q^{1+2\max(\sigma-1/2,1/\log X^{\beta_{j}})}}+O\Bigg(\sum_{X^{\beta_{i-1}}<q\leq X^{\beta_i}}\Big(\frac{\log q}{q\log X^{\beta_j}}+\frac{1}{q^2}\Big)\Bigg)\right)\nonumber\\
      \ll & \exp\left(\sum_{q\leq X^{\beta_j}}\frac{2h^2(q)}{q^{1+2\max(\sigma-1/2,1/\log X^{\beta_{j}})}}+O\Bigg(\sum_{q\leq X^{\beta_j}}\Big(\frac{\log q}{q\log X^{\beta_j}}+\frac{1}{q^2}\Big)\Bigg)\right)\nonumber\\
      \ll & \exp\left(\sum_{q\leq X^{\beta_j}}\frac{2h^2(q)}{q^{1+2\max(\sigma-1/2,1/\log X^{\beta_{j}})}}\right).
\end{align}
Here the last step, we use Lemma \ref{M} (b). Choose an appropriately large $N$. By the same arguments of getting (\ref{mainterm2}), we have
\begin{align}\label{sum2}
&\left(\frac{50a}{\big\lfloor e^N\beta^{-3/4}_{j+1}\big\rfloor}\right)^{2\big\lfloor\frac{1}{100\beta_{j+1}}\big\rfloor}\left(\sum_{\substack{n=\square\\ \Omega(n)=2\big\lfloor\frac{1}{100\beta_{j+1}}\big\rfloor\\ q|n\Rightarrow X^{\beta_j}<q\leq X^{\beta_{j+1}}}}\frac{(2\big\lfloor\frac{1}{100\beta_{j+1}}\big\rfloor)!\,\,t\big(n,X^{\beta_t}\big)}{\gamma(n) n^{1/2}}\right)\nonumber\\
\ll& \Big(\frac{e^{2N+1}}{10^3a^2\beta_{j+1}^{1/2}}\Big)^{-\big\lfloor \frac{1}{100\beta_{j+1}}\big\rfloor}\ll_{\bf a}
e^{-10(Q+1)a/\beta_j}.
\end{align}
Combining (\ref{sizeJ}), (\ref{Sj3}), (\ref{error2}), (\ref{Main}), (\ref{sum10}) and (\ref{sum2}), we obtain
\begin{align}\label{Sj4}
    &\sum_{\substack{(p,2)=1 \\p \in \mathcal{S}(j) }}\log p\big| L  \big(\sigma+it_1,  \chi^{(8p)} \big) \big|^{a_1} \cdots \big| L\big(\sigma+it_{k},\chi^{(8p)}  \big) \big|^{a_{k}}\Phi\Big(\frac {p}{X}\Big) \nonumber\\
 \ll & X\exp\left(\sum_{q\leq X^{\beta_j}}\frac{2h^2(q)}{q^{1+2\max(\sigma-1/2,1/\log X^{\beta_{j}})}}\right)\nonumber\\
 &\times\exp \left(\sum_{q\leq X^{\beta_j/2}} \frac{h(q^2)}{q^{1+2\max(\sigma-1/2,1/\log X^{\beta_j})}} \right) \exp \left(-\frac {(Q+1)a}{\beta_j} \right).
\end{align}
We know that $1-1/x<\log x$ for any $x\geq 1$. Observe that $|h(q)|\ll_{\bf a} 1$, by Lemma \ref{M} (b), we get
\begin{align}\label{h1}
    \sum_{q\leq X^{\beta_j}}\frac{2h^2(q)}{q^{1+2\max(\sigma-1/2,1/\log X^{\beta_{j}})}}&\leq \sum_{q\leq X^{\beta_j}}\frac{2h^2(q)}{q^{1+2/\log X^{\beta_j}}}=\sum_{q\leq X^{\beta_j}}\frac{2h^2(q)}{q}\left(1-\frac{q^{2/\log X^{\beta_j}}-1}{q^{2/\log X^{\beta_j}}}\right)\nonumber\\
    &=\sum_{q\leq X^{\beta_j}}\frac{2h^2(q)}{q}+O\left(\sum_{q\leq X^{\beta_j}}\frac{2h^2(q)}{q}\left|1-\frac{1}{q^{2/\log X^{\beta_j}}}\right|\right)\nonumber\\
    &=\sum_{q\leq X^{\beta_j}}\frac{2h^2(q)}{q}+O\left(\sum_{q\leq X^{\beta_j}}\frac{\log q}{q\log X^{\beta_j}}\right)\nonumber\\
    &=\sum_{q\leq X^{\beta_j}}\frac{2h^2(q)}{q}+O\left(1\right).
\end{align}
Using the same method, we can obtain
\begin{equation}\label{h2}
    \sum_{q\leq X^{\beta_j/2}} \frac{h(q^2)}{q^{1+2\max(\sigma-1/2,1/\log X^{\beta_j})}}\leq \sum_{q\leq X^{\beta_j/2}}\frac{h(q^2)}{q}+O\left(1\right).
\end{equation}
Inserting \eqref{h1} and \eqref{h2} into \eqref{Sj4} yields that
\begin{align}\label{Sj5}
&\sum_{\substack{(p,2)=1 \\p \in S(j) }} \log p\big| L\big(\sigma+it_1,\chi^{(8p)} \big) \big|^{a_1} \cdots \big| L\big(\sigma+it_{k},\chi^{(8p)}  \big) \big|^{a_{k}} \Phi \Big( \frac {p}{X} \Big) \nonumber\\
\ll & X \exp \Big (-\frac {(Q+1)a}{\beta_j} + \sum_{q \leq X}\frac {2h^2(q)}{q}+\sum_{q\leq X} \frac{h(q^2)}{q} \Big ).
\end{align}
   Through basic calculations, we get
\begin{align}\label{hexp}
& \sum_{q \leq X }\frac {(2h(q))^2}{2q}+\sum_{q\leq X} \frac{h(q^2)}{q} = \sum_{q \leq X }\frac {1}{2q}\Big ( \Big( \sum^{k}_{m=1}a_m\cos(t_m\log q) \Big)^2+\sum^{k}_{m=1}a_m\cos(2t_m\log q)\Big ) \nonumber\\
=& \sum_{q \leq X }\frac {1}{2q}\Big (\sum^{k}_{m=1}a^2_m\cos^2(t_m\log p)+2\sum_{1 \leq m<\ell \leq k}a_ma_{\ell}\cos(t_m\log q)\cos(t_{\ell}\log q)+\sum^{k}_{m=1}a_m\cos(2t_m\log q)\Big )\nonumber \\
=& \frac{1}{4}\sum^{k}_{m=1}a^2_m\sum_{q \leq X }\frac {1}{q}+\frac{1}{2}\sum_{1 \leq m<\ell \leq k}a_ma_{\ell}\sum_{q\leq X}\Big(\frac{\cos((t_m+t_{\ell})\log q)}{q}+\frac{\cos((t_m-t_{\ell})\log q)}{q}\Big )\nonumber\\
&+\sum^{k}_{m=1} \Big( \frac {a^2_m}{4}+\frac{a_{m}}{2} \Big) \sum_{q\leq X}\frac{\cos(2t_{m}\log q)}{q}.
\end{align}
Applying Lemma \ref{cos} to calculate (\ref{hexp}), we obtain
\begin{align}\label{hexp2}
     &\sum_{q \leq X }\frac {(2h(q))^2}{2q}+\sum_{q\leq X} \frac{h(q^2)}{q}\nonumber\\
     &\ll \frac{1}{4}\Big(\sum^{k}_{m=1}a^2_m\Big)\log\log X+\frac{1}{2}\sum_{1 \leq m<\ell \leq k}a_ma_{\ell}\log (|\zeta(1+1/\log X+i(t_m+t_{\ell}))||\zeta(1+1/\log X+i(t_m-t_{\ell}))|)\nonumber\\
    & \quad +\sum^{k}_{m=1} \Big( \frac {a^2_m}{4}+\frac{a_{m}}{2} \Big) \log|\zeta(1+1/\log X+i2t_{m})|.
\end{align}
Inserting (\ref{hexp2}) into (\ref{Sj5}), we have
\begin{align}\label{Sj6}
&\sum_{\substack{(p,2)=1 \\p \in S(j) }} \log p\big| L\big(\sigma+it_1,\chi^{(8p)} \big) \big|^{a_1} \cdots \big| L\big(\sigma+it_{k},\chi^{(8p)}  \big) \big|^{a_{k}} \Phi \Big( \frac {p}{X} \Big) \nonumber\\
\ll &  \exp \Big (-\frac {(Q+1)a}{\beta_j}\Big)X (\log X)^{(a_1^2+\cdots +a_{k}^2)/4}  \prod_{1\leq m<\ell \leq k} \big|\zeta(1+i(t_m-t_{\ell})+\frac 1{\log X}) \big|^{a_ma_{\ell}/2}\nonumber\\
   &\times \big|\zeta(1+i(t_m+t_{\ell})+\frac 1{\log X}) \big|^{a_ma_{\ell}/2}\prod_{1\leq m\leq k} \big|\zeta(1+2it_{m}+\frac 1{\log X}) \big|^{a^2_m/4+a_{m}/2}.
\end{align}
Since
\begin{equation*}
    \sum_{j=1}^{\mathcal{J}}e^{-\frac {(Q+1)a}{\beta_j}}\ll 1,
\end{equation*}
we derive
\begin{align*}
&\sum_{j=1}^{\mathcal{J}}  \sum_{\substack{(p,2)=1 \\p \in \mathcal{S}(j) }}\log p\big| L\big(\sigma+it_1,\chi^{(8p)} \big) \big|^{a_1} \cdots \big| L\big(\sigma+it_{k},\chi^{(8p)}  \big) \big|^{a_{k}}  \Phi \Big( \frac p{X} \Big) \nonumber\\
   \ll & X (\log X)^{(a_1^2+\cdots +a_{k}^2)/4}  \prod_{1\leq m<\ell \leq k} \big|\zeta(1+i(t_m-t_{\ell})+\frac 1{\log X}) \big|^{a_ma_{\ell}/2}\nonumber\\
   &\times \big|\zeta(1+i(t_m+t_{\ell})+\frac 1{\log X}) \big|^{a_ma_{\ell}/2}\prod_{1\leq m\leq k} \big|\zeta(1+2it_{m}+\frac 1{\log X}) \big|^{a^2_m/4+a_{m}/2},
\end{align*}
which completes the proof of Theorem \ref{t1}.

\section{Proof of Theorem 1.2}

By Lemma \ref{Y>X}, we can see that
\begin{align*}
S_m(X,Y) \ll \log X\sum_{\substack{d \leq X \\ (d,2)=1}}\Big | \sum_{n \leq Y}\chi^{(8d)}(n)\Big |^{2m}\ll X^{m+1}\log X,
\end{align*}
for any $m>0$.
Therefore, we derive the result of Theorem \ref{t2} with $Y\geq X$.
In the remaining part of the proof, we assume that $Y \leq X$.

Define a non-negative smooth function
\begin{equation}\label{PhiU}
\Phi_U(t)\begin{cases}
            \leq 1,\quad  x\in [0,1/U]\cup [1-1/U,1],\\
            =1,\quad  x\in[1/U,1-1/U],\\
            =0,\quad \text{otherwise },
            \end{cases}
\end{equation}
satisfying $\Phi^{(j)}_U(t) \ll_j U^j$ for all integers $j \geq 0$, where $U$ is a parameter, we will clarify it later. By Mellin transform and we repeat integration by parts, then for any integer $A \geq 1$ and $\Re(s) \geq 1/2$, we have
\begin{align}
\label{whatbound}
 \widehat{\Phi}_U(s)=\int_0^\infty \Phi_U(t)t^{s-1}dt  \ll  U^{A-1}(1+|s|)^{-A}.
\end{align}
Using the smooth function $\Phi_U(\frac nY)$, we get the upper bound of $S_m(X,Y)$:
\begin{align}\label{SmXY}
S_m(X,Y) \ll S_m^{(1)}(X,Y)+S_m^{(2)}(X,Y),
\end{align} 	
where
\begin{equation*}
   S_m^{(1)}(X,Y):= \sum_{2<p \leq X }\log p\Big | \sum_{n}\chi^{(8p)}(n)\Phi_U \Big( \frac {n}{Y} \Big)\Big |^{2m}
\end{equation*}
and
\begin{equation*}
    S_m^{(2)}(X,Y):=\sum_{2<p \leq X }\log p\bigg|\sum_{n\leq Y} \chi^{(8p)}(n)\Big(1-\Phi_U \Big( \frac {n}{Y} \Big) \Big)\bigg|^{2m}.
\end{equation*}

\begin{proposition}\label{Sm1}
With the notation as above and assume the truth of GRH. We have for any integer $k \geq 1$ and any real numbers $2m \geq k+1, \varepsilon>0$
    \begin{align*}
    S_m^{(1)}(X,Y)\ll XY^m(\log X)^{ R(m,k,\varepsilon)}.
    \end{align*}
\end{proposition}

\begin{proposition}\label{Sm2}
    With the notation as above and assume the truth of GRH. We have for $m \geq 1$,
\begin{equation*}
S_m^{(2)}(X,Y) \ll XY^m.
\end{equation*}
\end{proposition}
By (\ref{SmXY}), Proposition \ref{Sm1} and Proposition \ref{Sm2},  we complete the proof of Theorem \ref{t2}. In next section, we will give the proofs of Proposition \ref{Sm1} and Proposition \ref{Sm2}.

\section{Proofs of Propositions }

\subsection{Proof of Proposition \ref{Sm1}}
To deal with $S_m^{(1)}(X,Y)$, we apply Mellin inversion. Thanks to (\ref{whatbound}) and Lemma \ref{L(s,chi)}, we can shift the line of integration from $\Re(s)=2$ to $\Re(s)=1/2$,
\begin{align}\label{SmXY1}
 &S_m^{(1)}(X,Y)
 =\sum_{2<p\leq X }\log p\Big | \int\limits_{(2)}L(s, \chi^{(8p)})Y^s\widehat{\Phi}_U(s)  ds\Big |^{2m}
= \sum_{2<p\leq X }\log p\Big | \int\limits_{(1/2)}L(s, \chi^{(8p)})Y^s\widehat{\Phi}_U(s)  ds\Big |^{2m}\nonumber\\
\ll&\sum_{2<p \leq X }\log p\Big | \int\limits_{\substack{ (1/2) \\ |\Im(s)| \leq X^V}}L(s, \chi^{(8p)})Y^s\widehat{\Phi}_U(s)  ds\Big |^{2m}+
\sum_{2<p\leq X }\log p\Big | \int\limits_{\substack{ (1/2) \\ |\Im(s)| > X^V}}L(s, \chi^{(8p)})Y^s\widehat{\Phi}_U(s)  ds\Big |^{2m}.
\end{align}
Here the last step, we separate the integral into two parts, for some $V>0$ which will be specified later.  By H\"older's inequality, we have
\begin{align}
\label{SmXY11}
  &\sum_{2<p \leq X} \log p \Big | \int\limits_{\substack{ (1/2) \\ |\Im(s)| > X^V}}L(s, \chi^{(8p)})Y^s\widehat{\Phi}_U(s)  ds\Big |^{2m}\nonumber\\
  \ll & Y^m \sum_{2<p \leq X} \log p \Big(\int\limits_{\substack{ (1/2) \\ |\Im(s)| > X^V}}| L(s, \chi^{(8p)})|\Big | \widehat{\Phi}_U(s) \Big| ^{\frac{1}{2m}}\Big | \widehat{\Phi}_U(s) \Big| ^{\frac{2m-1}{2m}}|ds|\Big)^{2m}\nonumber\\
 \ll & Y^m\Big (\int\limits_{\substack{ (1/2) \\ |\Im(s)| > X^V}}\Big | \widehat{\Phi}_U(s) \Big| |ds| \Big )^{2m-1}
 \int\limits_{\substack{ (1/2) \\ |\Im(s)| > X^V}}\sum_{2<p\leq X }\log p\Big |L(s, \chi^{(8p)})\Big|^{2m} \Big| \widehat{\Phi}_U(s)\Big | |ds|
\end{align} 	
for $m \geq 1/2$.
By \eqref{whatbound}, we obtain
\begin{align}\label{phiint}
  \int\limits_{\substack {(1/2) \\ |\Im(s)| > X^V}}\Big | \widehat{\Phi}_U(s) \Big| | ds|
 \ll & \int\limits_{\substack {(1/2) \\ |\Im(s)| > X^V}}\frac U{1+|s|^2} |ds| \ll_V \frac {U}{X^V}.
\end{align}
Applying Lemma \ref{L(s,chi)} and (\ref{whatbound}) to estimate the inner integral of (\ref{SmXY11}), we get
\begin{align}
\label{>V}
  & \int\limits_{\substack {(1/2) \\ |\Im(s)| > X^V}}\sum_{2<p\leq X }\log p\Big |L(s, \chi^{(8p)})\Big|^{2m}\cdot \Big |\widehat{\Phi}_U(s)\Big | |ds| \nonumber\\
  \ll & \int\limits_{\substack {(1/2) \\ |\Im(s)| > X^V}}\sum_{2<p\leq X }|ps|^{\varepsilon}\cdot\frac U{1+|s|^2}  | ds| \ll XUX^{-V(1-\varepsilon)+\varepsilon}.
\end{align}
   We choose
   \begin{equation}\label{UV}
   U=X^{3\varepsilon}, \quad V=6\varepsilon
   \end{equation}
   to ensure the term (\ref{SmXY11}) $\ll XY^m$ with $m\geq 1/2$. Therefore, we get
\begin{align*}
S^{(1)}_m(X,Y)
\ll &
   \sum_{2<p \leq X }\log p \Big | \int\limits_{\substack{ (1/2) \\ |\Im(s)| \leq X^{\varepsilon}}}L(s, \chi^{(8p)})Y^s\widehat{\Phi}_U(s)  ds\Big |^{2m}
   +O(XY^m)\nonumber \\
   \ll & Y^m \sum_{2<p \leq X }\log p
   \Big | \int\limits_{\substack{ |t| \leq X^{\varepsilon}}}\Big |L( \tfrac{1}{2}+it, \chi^{(8p)})\Big |\frac 1{1+|t|}  dt\Big |^{2m}
   +O(XY^m).
\end{align*}

Proposition \ref{Sm1} follows from the following Lemma.
\begin{lemma}\label{fdiff}
With the notation as above and assume the truth of GRH. We have for any integer $k \geq 1$ and any real numbers $2m \geq k+1, \varepsilon>0$,
\begin{equation*}
\sum_{2<p \leq X }\log p
   \Big | \int\limits_{\substack{ |t| \leq X^{\varepsilon}}}\Big |L(1/2+it, \chi^{(8p)})\Big |\frac 1{1+|t|}dt\Big |^{2m} \ll X(\log X)^{ R(m,k,\varepsilon)},
\end{equation*}
where $R(m,k,\varepsilon)$ is defined in \eqref{Edef}.
\end{lemma}
\begin{proof}
By symmetry and H\"older's inequality for $a=1-1/(2m)+\varepsilon$ with $\varepsilon>0$, we get
\begin{align}\label{lemma51}
 &\sum_{2<p \leq X }\log p\Big | \int_{ |t| \leq X^{\varepsilon}}   \Big |L(\tfrac{1}{2}+it, \chi^{(8p)})|\frac {d t}{|t|+1} \Big |^{2m} \nonumber\\
 \ll & \sum_{2<p \leq X }\log p\Big |\int_0^{X^{\varepsilon}} \frac{|L(\tfrac{1}{2}+it,\chi^{(8p)})|}{t+1} d t\Big |^{2m} \nonumber\\
 \ll & \sum_{2<p \leq X }\log p\Big |\sum_{n\leq \log X+1}\int_{e^{n-1}-1}^{e^n-1} \frac{|L(\tfrac{1}{2}+it,\chi^{(8p)})|}{t+1} d t\Big |^{2m}\nonumber\\
  \leq & \sum_{2<p \leq X }\log p\bigg(\sum_{n\leq \log X+1} n^{-2am/(2m-1)} \bigg)^{2m-1}
    \sum_{n\leq  \log X+1} \bigg(n^a\int_{e^{n-1}-1}^{e^{n}-1 } \frac{|L(\tfrac{1}{2}+it,\chi^{(8p)}) |}{t+1} d t\bigg)^{2m}   \nonumber\\
   \ll & \sum_{n\leq  \log X+1} \frac{n^{2m-1+\varepsilon} }{e^{2nm} } \sum_{2<p \leq X }\log p\bigg( \int_{e^{n-1}-1}^{e^{n}-1 } |L(\tfrac{1}{2}+it,\chi^{(8p)}) | d t \bigg)^{2m}.
\end{align}
Below we will estimate the summation
\begin{equation*}
    T_m(E,X):=\sum_{2<p \leq X }\log p\bigg( \int_{0}^{E} |L(\tfrac{1}{2}+it,\chi^{(8p)}) | d t \bigg)^{2m},
\end{equation*}
where $10 \leq E=X^{O(1)}$. The method of getting the upper bound of $T_m(E,X)$ is similar to the proof of (6.1) of Szab\'o \cite{Szab}, so we will not explain it in much detail.

First, we put out $k$ integrals, which we can do because $2m\geq k+1$ by assumption. Using the notation $d {\mathbf t} =dt_1\cdots dt_k$, we obtain
\begin{align}\label{Lintdecomp}
   T_m(E,X)
      \ll \sum_{2<p \leq X }\log p\int_{[0,E]^k}\prod_{a=1}^k|L(1/2+ it_a, \chi^{(8p)})| \bigg(\int_{\mathcal{D} }|L(1/2+iu, \chi^{(8p)})| d u \bigg)^{2m-k} d\mathbf{t},
\end{align}
where $\mathcal{D}=\mathcal{D}(t_1,\ldots,t_k)=\{ u\in [0,E]:|t_1-u|\leq |t_2-u|\leq \cdots \leq |t_k-u| \}$. Here, we restrict the integration over $\mathcal{D}$ by symmetry.

We set
\begin{align*}
    &\mathcal{B}_1=\big[-\frac{1}{\log X},\frac{1}{\log X}\big],\quad \quad \quad \mathcal{B}_H=[-E,E]\setminus \bigcup_{1\leq h<H} \mathcal{B}_h \\
    &\mathcal{B}_h=\big[-\frac{e^{h-1}}{\log X}, -\frac{e^{h-2}}{\log X}\big]
  \cup \big[\frac{e^{h-2}}{\log X}, \frac{e^{h-1}}{\log X}\big], \quad \text{for}\quad 2\leq h< \lfloor \log \log X\rfloor+10 =: H.
\end{align*}
For any $t_1\in [0,E]$,  we have
\begin{equation*}
    \mathcal{D}\subset [0,E] \subset t_1+[-E,E]\subset \bigcup_{1\leq h\leq H} t_1+\mathcal{B}_h.
\end{equation*}
Thus, if we denote $\mathcal{A}_h=\mathcal{B}_h\cap (-t_1+\mathcal{D})$, then $(t_1+\mathcal{A}_h)_{1\leq h\leq H}$ form a partition of $\mathcal{D}$. Observe that $|\mathcal{A}_h|\leq |\mathcal{B}_h|$.
Applying H?lder's inequality twice we get
\begin{align}\label{LintoverD}
    & \bigg(\int_{\mathcal{D}}|L(1/2 + iu, \chi^{(8p)})| d u\bigg)^{2m-k}  \leq \bigg( \sum_{1\leq h\leq H} \frac{1}{h}\cdot  h \int_{t_1+\mathcal{A}_h} |L( 1/2+iu, \chi^{(8p)})| du  \bigg)^{2m-k} \nonumber\\
     \leq & \bigg(\sum_{1\leq h\leq H} h^{2m-k} \bigg( \int_{t_1+\mathcal{A}_h} \big|L( 1/2+iu, \chi^{(8p)})\big| d u  \bigg)^{2m-k}\bigg)
     \bigg(\sum_{1\leq h\leq K } h^{-(2m-k)/(2m-k-1)} \bigg)^{2m-k-1} \nonumber\\
     \ll & \sum_{1\leq h\leq H} h^{2m-k} \bigg( \int_{t_1+\mathcal{A}_h} |L( 1/2+iu, \chi^{(8p)})| d u \bigg)^{2m-k} \nonumber\\
     \leq & \sum_{1\leq h\leq H} h^{2m-k} |\mathcal{B}_h|^{2m-k-1} \int_{t_1+\mathcal{A}_h} |L(1/2+iu, \chi^{(8p)})|^{2m-k} d u.
\end{align}
For $\mathbf{t}=(t_1,\ldots,t_k)$, we write
\begin{equation}\label{Lnew}
    L(\mathbf{t},u)=\sum_{\substack{p \leq X \\ (p,2)=1}}\log p\prod_{a=1}^k|L( 1/2+it_a, \chi^{(8p)})| \cdot |L( 1/2+iu, \chi^{(8p)})|^{2m-k}.
\end{equation}
 From \eqref{Lintdecomp}--\eqref{Lnew}, we obtain
\begin{align}\label{Lintest}
    T_m(E,X)\ll&
    \sum_{1\leq h\leq H} h^{2m-k} |\mathcal{B}_{h}|^{2m-k-1} \int_{[0,E]^k}\int_{t_1+\mathcal{A}_{h}} L(\mathbf{t},u) \,du\, d \mathbf{t}  \nonumber\\
     \ll & \sum_{1\leq h_0, h_1, \ldots h_{k-1}\leq H} h_0^{2m-k} |\mathcal{B}_{h_0}|^{2m-k-1} \int_{\mathcal{C}_{h_0,h_1, \cdots, h_{k-1}}} L(\mathbf{t},u) \,d u \,d \mathbf{t},
\end{align}
where
$$\mathcal{C}_{h_0,h_1, \cdots, h_{k-1}}=\{(t_1,\ldots,t_k,u)\in [0,E]^{k+1}: u\in t_1+ \mathcal{A}_{h_0},\, |t_{i+1}-u|-|t_i-u|\in \mathcal{B}_{h_i}, \ 1 \leq i \leq k-1\}.$$
Using Corollary \ref{cor1} to deal with $L(\mathbf{t},u)$, for $(t_1,\ldots,t_k,u)\in \mathcal{C}_{h_0,h_1, \cdots, h_{k-1}}$, we have
\begin{align}\label{L(t,u)}
&L(\mathbf{t},u)\nonumber\\
\ll& X(\log X)^{\frac{k+(2m-k)^2}{4}}\prod_{1\leq i<j\leq k} g(|t_i-t_{j}|)^{a_ia_{j}/2}g(|t_i+t_{j}|)^{a_ia_{j}/2}\prod_{1\leq i\leq k} g(|2t_i|)^{a^2_i/4+a_i/2}.
\end{align}
Below, we provide the upper bound of $g(\alpha)$ for two cases, with $\alpha$ in different ranges,  and deduce the corresponding upper bound of (\ref{Lintest}).

\textbf{Case 1:} $h_0<H$.

By the definition of $\mathcal{C}_{h_0,h_1, \cdots, h_{k-1}}$ and the observation that $t_i, u \geq 0, 1\leq i \leq k$, we have
\begin{equation*}
     \frac{e^{h_0}}{\log X}\ll |t_1-u|\ll |t_1+u|\ll E =X^{O(1)}.
\end{equation*}
 Recalling the definition of $g$ in \eqref{gDef}, we derive
 \begin{equation*}
     g(|t_1\pm u|)\ll \frac{\log X}{e^{h_0}} \log \log E.
 \end{equation*}
Since the definition of $\mathcal{A}_j$, we know that $|t_2-u|\geq |t_1-u|$, furthermore
\begin{equation*}
    \frac{e^{h_0}}{\log X}+ \frac{e^{h_1}}{\log X}\ll |t_2-u|= |t_1-u|+(|t_2-u|-|t_1-u|)\ll |t_2+u|\ll E=X^{O(1)},
\end{equation*}
which implies that
\begin{equation*}
    g(|t_2\pm u|)\ll \frac{\log X}{e^{\max(h_0,h_1) }} \log \log E.
\end{equation*}
Similarly, for any $1 \leq i \leq k$, we have
\begin{equation}\label{gtiu}
    g(|t_i\pm u|)\ll \frac{\log X}{e^{\max(h_0,h_1,\ldots, h_{i-1})} }\log \log E.
\end{equation}
Moreover, for any $1 \leq i < j \leq k$, we have
\begin{equation*}
    \sum^{j-1}_{s=i}(|t_{s+1}-u|-|t_s-u|) \leq |t_j-t_i| \leq |t_j+t_i|,
\end{equation*}
then
\begin{equation}\label{gtjti}
    g(|t_{j}\pm t_i|)\ll \frac{\log X }{e^{\max(h_i,\ldots, h_{j-1} ) }} \log \log E.
\end{equation}
Trivially, from the definition of $g$ in (\ref{gDef}), we get
\begin{equation}\label{g2ug2ti}
   g(|2u|)\ll \log X,\quad g(|2t_i|)\ll \log X,\,\,\,\,\text{for any}\,\,\,1\leq i \leq k.
\end{equation}
By (\ref{L(t,u)})--(\ref{g2ug2ti}), we calculate that
\begin{align}\label{L(t,u)1}
     & L(\mathbf{t},u) \nonumber\\
     \ll & X(\log X)^{\frac{(2m-k)^2+k}{4}+\frac{(2m-k)^2}{4}+\frac{2m-k}{2}+\frac{3k}{4}}(\log \log E)^{\frac{k(k-1)}{2}+k(2m-k)}\nonumber\\
     &\quad\quad\quad\quad\quad\quad\quad\times\bigg(\prod^{k-1}_{i=0}\frac{\log X}{e^{ \max(h_0,h_1,\ldots, h_{i}) }} \bigg)^{2m-k}
     \bigg(\prod^{k-1}_{i=1} \prod^{k}_{j=i+1}\frac{\log X}{e^{\max(h_i,\ldots, h_{j-1} ) } } \bigg) \nonumber\\
     = & X(\log X)^{m(2m+1)}(\log \log E)^{\frac{k(k-1)}{2}+k(2m-k)} \\ \nonumber
     &\quad\quad\quad\times\exp\Big( -(2m-k)\sum^{k-1}_{i=0}\max(h_0,h_1,\ldots, h_{i})-\sum^{k-1}_{i=1} \sum^{k}_{j=i+1}\max(h_i,\ldots, h_{j-1} )\Big).
\end{align}
The volume of the region $\mathcal{C}_{h_0,h_1, \cdots, h_{k-1}}$ is $\ll  E^k e^{h_0+h_1+\cdots+h_{k-1}} (\log X)^{-k}$, and $ |\mathcal{B}_{h_0}|\ll e^{h_0}/\log X$. Inserting (\ref{L(t,u)1}) into (\ref{Lintest}), we get
\begin{align}\label{case1}
       &  \sum_{\substack{1\leq h_0<H\\ 1\leq h_1, \ldots h_{k-1}\leq H}}  h_0^{2m-k} |\mathcal{B}_{h_0}|^{2m-k-1} \int\limits_{\mathcal{C}_{h_0,h_1, \cdots, h_{k-1}}} L(\mathbf{t},u) \,d u \,d\mathbf{t} \nonumber\\
    \ll & X(\log X)^{2m^2-m+1}E^k(\log \log E)^{\frac{k(k-1)}{2}+k(2m-k)}\sum_{\substack{1\leq h_0<H\\ 1\leq h_1, \ldots h_{k-1}\leq H}}  h_0^{2m-k} \nonumber\\
    &  \times \exp\Big( (2m-k-1)h_0+\sum^{k-1}_{i=0}h_i-(2m-k)\sum^{k-1}_{i=0}\max(h_0,h_1,\ldots, h_{i})-\sum^{k-1}_{i=1} \sum^{k}_{j=i+1}\max(h_i,\ldots, h_{j-1} )\Big) \nonumber\\
    = & X(\log X)^{2m^2-m+1}E^k(\log \log E)^{\frac{k(k-1)}{2}+k(2m-k)} \nonumber\\
    &  \times \sum_{\substack{1\leq h_0<H\\ 1\leq h_1, \ldots h_{k-1}\leq H}}  h_0^{2m-k}\exp\Big( -(2m-k)\sum^{k-1}_{i=1}\max(h_0,h_1,\ldots, h_{i})-\sum^{k-1}_{i=1} \sum^{k}_{j=i+2}\max(h_i,\ldots, h_{j-1} )\Big) \nonumber\\
    \ll &   X(\log X)^{2m^2-m+1}E^k(\log \log E)^{\frac{k(k-1)}{2}+k(2m-k)} \nonumber\\
    &\quad\quad\quad\times\sum_{\substack{1\leq h_0<H \\ 1\leq h_1, \ldots h_{k-1}\leq H}}  h_0^{2m-k}\exp\Big( -(2m-k)\sum^{k-1}_{i=1}\frac {\sum^{i}_{s=0}h_s}{i+1}-\sum^{k-1}_{i=1} \sum^{k}_{j=i+2}\frac {\sum^{j-1}_{s=i}h_s}{j-i}\Big)  \nonumber\\
    \ll &   X(\log X)^{2m^2-m+1}E^k(\log \log E)^{\frac{k(k-1)}{2}+k(2m-k)}.
\end{align}

\textbf{Case 2:} $h_0=H$

(1) $|u|> 5$.

For this case, we have $e^{h_0}\ll \log X$. Hence, for any $1 \leq i \leq k$, we get
\begin{equation}\label{gtiu1}
    g(|t_i\pm u|)\ll \log\log  E.
\end{equation}
Regarding the estimate of $g(|t_j\pm t_i|)$, it is the same as (\ref{gtjti}).  We estimate $g(|2t_i|)$ by the trivial bound, and since $|t_1-u| \leq |t_1|+|u|$, $|u| \geq 5$, we have
\begin{equation}\label{g2ug2ti1}
   g(|2u|) \ll \log\log  E,\quad\quad g(|2t_i|) \ll \log X,\,\,\text{for any}\,\,1\leq i\leq k.
\end{equation}
By (\ref{L(t,u)}), (\ref{gtjti}), (\ref{gtiu1}) and (\ref{g2ug2ti1}), we derive
\begin{align}\label{L(t,u)2}
     &L(\mathbf{t},u) \nonumber\\
     \ll & X(\log X)^{\frac{(2m-k)^2+k}{4}+\frac{3k}{4}}(\log \log E)^{\frac{(2m-k)^2}{4}+(2m-k)(k+1/2)+\frac{k(k-1)}{2}}
     \bigg(\prod^{k-1}_{i=1} \prod^{k}_{j=i+1}\frac{\log X}{e^{\max(h_i,\ldots, h_{j-1} ) } } \bigg) \nonumber\\
     = & X(\log X)^{\frac{(2m-k)^2}{4}+\frac{k(k+1)}{2}}(\log \log E)^{\frac{(2m-k)^2}{4}+(2m-k)(k+1/2)+\frac{k(k-1)}{2}}  \nonumber\\
&\quad\quad\quad\quad\quad\quad\quad\quad\quad\quad\quad\quad\quad\quad\quad\times\exp\Big( -\sum^{k-1}_{i=1} \sum^{k}_{j=i+1}\max(h_i,\ldots, h_{j-1} )\Big).
\end{align}
The volume of the region $\mathcal{C}_{H,h_1, \cdots, h_{k-1}}$ is
\begin{equation}\label{Vol}
    \ll E^{k+1} e^{h_1+\cdots+h_{k-1}} (\log X)^{-k+1}.
\end{equation}
As $|\mathcal{B}_H|\ll E$, we deduce from (\ref{Lintest}), (\ref{L(t,u)2}) and (\ref{Vol}) that
\begin{align}\label{case21}
     &\sum_{\substack{1\leq h_1, \ldots h_{k-1}\leq H}}  H^{2m-k} |\mathcal{B}_{H}|^{2m-k-1} \int\limits_{\substack{\mathcal{C}_{H,h_1, \cdots, h_{k-1}} \\ |u| \geq 5}} L(\mathbf{t},u) \,d u\, d \mathbf{t}  \nonumber\\
     \ll & X(\log X)^{\frac{(2m-k)^2}{4}+\frac{k(k-1)}{2}+1}E^{2m}(\log\log X)^{2m-k} (\log \log E)^{\frac{(2m-k)^2}{4}+(2m-k)(k+1/2)+\frac{k(k-1)}{2}}\nonumber\\
     & \hspace*{2cm} \times \sum_{\substack{1\leq h_1, \ldots h_{k-1}\leq H}} \exp\Big( \sum^{k-1}_{i=1}h_i-\sum^{k-1}_{i=1} \sum^{k}_{j=i+1}\max(h_i,\ldots, h_{j-1} )\Big) \nonumber\\
     \ll &   X(\log X)^{\frac{(2m-k)^2}{4}+\frac{k(k-1)}{2}+1}E^{2m}(\log\log X)^{2m-k} (\log \log E)^{\frac{(2m-k)^2}{4}+(2m-k)(k+1/2)+\frac{k(k-1)}{2}}.
\end{align}

(2) $|u|\leq 5$.
We estimate $g(|2u|)$ by the trivial bound, and since $e^{10}<|t_1-u| \leq |t_i-u| \leq |t_i|+|u|$, $|u|\leq 5$, for any $1 \leq i \leq k$, we have $|t_i|\geq 5$ and
\begin{equation}\label{g2ug2ti2}
   g(|2u|) \ll \log X,\quad\quad g(|2t_i|) \ll \log \log E.
\end{equation}
By (\ref{L(t,u)}), (\ref{gtjti}), (\ref{gtiu1}) and (\ref{g2ug2ti2}), we get
\begin{align}\label{L(t,u)3}
 &L(\mathbf{t},u) \nonumber\\
 \ll & X(\log X)^{\frac{(2m-k)^2+k}{4}+\frac{(2m-k)^2}{4}+\frac{2m-k}{2}}(\log \log E)^{(2m-k)k+\frac{k(k-1)}{2}+\frac{3k}{4}}
     \bigg(\prod^{k-1}_{i=1} \prod^{k}_{j=i+1}\frac{\log X}{e^{\max(h_i,\ldots, h_{j-1} ) } } \bigg) \nonumber\\
     = & X(\log X)^{m(2m+1)-k(2m+3/4-k)}(\log \log E)^{(2m-k)k+\frac{k(k-1)}{2}+\frac{3k}{4}}  \nonumber\\
&\quad\quad\quad\quad\quad\quad\quad\quad\quad\quad\quad\quad\quad\quad\quad\times\exp\Big( -\sum^{k-1}_{i=1} \sum^{k}_{j=i+1}\max(h_i,\ldots, h_{j-1} )\Big).
\end{align}

   As we have $|t_i| \geq 5$ for any $1 \leq i \leq k$ when $|u| \leq 5$, we see that the volume of the region $\mathcal{C}_{H,h_1, \cdots, h_{k-1}}$ is $\ll E^{k}$. It follows that
\begin{align}\label{case22}
 &\sum_{\substack{1\leq h_1, \ldots h_{k-1}\leq H}}  H^{2m-k} |\mathcal{B}_{H}|^{2m-k-1} \int\limits_{\substack{\mathcal{C}_{H,h_1, \cdots, h_{k-1}} \\ |u| \leq 5}} L(\mathbf{t},u)\, d u\, d \mathbf{t}  \nonumber\\
     \ll & X(\log X)^{m(2m+1)-k(2m+3/4-k)}E^{2m-1}(\log\log X)^{2m-k}(\log \log E)^{(2m-k)k+\frac{k(k-1)}{2}+\frac{3k}{4}} \nonumber\\
     & \hspace*{2cm} \times \sum_{\substack{1\leq h_1, \ldots h_{k-1}\leq H}} \exp\Big( -\sum^{k-1}_{i=1} \sum^{k}_{j=i+1}\max(h_i,\ldots, h_{j-1} )\Big) \nonumber\\
     \ll &   X(\log X)^{m(2m+1)-k(2m+3/4-k)}E^{2m-1}(\log\log X)^{2m-k}(\log \log E)^{(2m-k)k+\frac{k(k-1)}{2}+\frac{3k}{4}}.
\end{align}
From (\ref{Lintest}), (\ref{case1}), (\ref{case21}) and (\ref{case22}), we obtain
\begin{align}\label{lemma50}
    &T_m(E,X)\nonumber\\
    \ll& X\Bigg(
    (\log X)^{2m^2-m+1}E^k(\log \log E)^{\frac{k(k-1)}{2}+k(2m-k)}\nonumber\\
    &+(\log X)^{\frac{(2m-k)^2}{4}+\frac{k(k-1)}{2}+1}E^{2m}(\log\log X)^{2m-k} (\log \log E)^{\frac{(2m-k)^2}{4}+(2m-k)(k+1/2)+\frac{k(k-1)}{2}}\nonumber\\
    &\quad\quad+(\log X)^{m(2m+1)-k(2m+3/4-k)}E^{2m-1}(\log\log X)^{2m-k}(\log \log E)^{(2m-k)k+\frac{k(k-1)}{2}+\frac{3k}{4}}\Bigg).
\end{align}
We apply (\ref{lemma50}) to estimate (\ref{lemma51}) for any integer $k \geq 1$ and any real numbers $2m \geq k+1, \varepsilon>0$, then
\begin{align*}
  &\sum_{n\leq  \log X+1}  \frac{n^{2m-1+\varepsilon} }{e^{2nm} }\sum_{\substack{p \leq X \\ (p,2)=1}} \log p\bigg( \int_{e^{n-1}-1}^{e^{n}-1 } |L(1/2+it,\chi^{(8p)}) | dt \bigg)^{2m}
    \\
     \ll & X\sum_{n\leq  \log X+1} \frac{n^{2m-1+\varepsilon} }{e^{2nm} } \\
    & \times
     \Big(
    (\log X)^{2m^2-m+1}e^{kn}(\log \log E)^{\frac{k(k-1)}{2}+k(2m-k)}\nonumber\\
    &\,\,+(\log X)^{\frac{(2m-k)^2}{4}+\frac{k(k-1)}{2}+1}e^{2mn}(\log\log X)^{2m-k} (\log \log E)^{\frac{(2m-k)^2}{4}+(2m-k)(k+1/2)+\frac{k(k-1)}{2}}\nonumber\\
    &\quad\quad+(\log X)^{m(2m+1)-k(2m+3/4-k)}e^{(2m-1)n}(\log\log X)^{2m-k}(\log \log E)^{(2m-k)k+\frac{k(k-1)}{2}+\frac{3k}{4}}\Big)\nonumber\\
    &\ll X(\log X)^{\varepsilon}\Big(
    (\log X)^{2m^2-m+1}
    +(\log X)^{\frac{(2m-k)^2}{4}+\frac{k(k-1)}{2}+2m+1}\nonumber\\
    &\quad\quad\quad\quad\quad\quad\quad\quad+(\log X)^{m(2m+1)-k(2m+3/4-k)}\Big)\\
    &\ll X(\log X)^{R(m,k,\varepsilon)}.
\end{align*}
We complete the proof of Lemma \ref{fdiff}.
\end{proof}

\subsection{Proof of Proposition \ref{Sm2}}

We apply the Cauchy-Schwarz inequality to see that
\begin{align}\label{pocs1}
   S_m^{(2)}(X,Y)
    \leq  \bigg(\sum_{\substack{p \leq X \\ (p,2)=1}}\log p&\bigg|\sum_{n\leq Y} \chi^{(8p)}(n)\Big (1-\Phi_U \Big( \frac {n}{Y} \Big) \Big )\bigg|^{2}\bigg)^{\frac{1}{2}}\nonumber\\
    &\times\bigg(\sum_{\substack{p \leq X \\ (p,2)=1}}\log p\bigg|\sum_{n\leq Y} \chi^{(8p)}(n)\Big (1-\Phi_U \Big( \frac {n}{Y} \Big)\Big )\bigg|^{4m-2}\bigg)^{\frac{1}{2}}.
\end{align}
Using Lemma \ref{Heath} with $Z=Y$ and $Y \leq X$ to the first term of product (\ref{pocs1}), it yields that
\begin{align*}
  &\sum_{\substack{p \leq X \\ (p,2)=1}}\log p\bigg|\sum_{n\leq Y} \chi^{(8p)}(n) \Big (1-\Phi_U \Big( \frac {n}{Y} \Big)\Big ) \bigg|^{2} \ll \log X \sum_{\substack{d \leq X \\ (d,2)=1}}\bigg|\sum_{n\leq Y} \chi^{(8d)}(n) \Big (1-\Phi_U \Big( \frac {n}{Y} \Big)\Big ) \bigg|^{2}\\
  \ll & (XY)^{\varepsilon}(X+Y)\sum_{\substack{n_1, n_2 \leq Y \\ n_1n_2=\square} }\Big (1-\Phi_U \Big( \frac {n_1}{Y} \Big)\Big )\Big (1-\Phi_U \Big(\frac {n_2}{Y} \Big) \Big ) \\
  \ll &
  X^{1+\varepsilon}\sum_{\substack{ n_1, n_2 \in[1,Y/U]\cup [Y(1-1/U),Y] \\ n_1n_2=\square} }1 .
\end{align*}
  We write $n_1=d m^2_1, n_2=d m^2_2$ with $d$ square.  The above is
\begin{align}\label{pocs2}
  \ll& X^{1+\varepsilon}\sum_{\substack{d \leq Y}}\sum_{\substack{ m_1, m_2 \in[1,(Y/(dU))^{1/2}]\cup[(Y(1-1/U)/d)^{1/2}, (Y/d)^{1/2}]}}1 \nonumber\\
  \ll & X^{1+\varepsilon}\sum_{\substack{d \leq Y}} \Big(Y/(dU)+\Big((Y/d)^{1/2}-(Y(1-1/U)/d)^{1/2}\Big)^2 \Big)
\nonumber\\
   \ll & X^{1+\varepsilon}YU^{-1}\ll X^{1-2\varepsilon}Y
\end{align}
with $U=X^{3\varepsilon}$.
For the second term of product (\ref{pocs1}), we have the upper bound
\begin{align}\label{pocs3}
&\sum_{\substack{p \leq X \\ (p,2)=1}}\log p\bigg|\sum_{n\leq Y} \chi^{(8p)}(n)\Big (1-\Phi_U \Big( \frac {n}{Y} \Big)\Big )\bigg|^{4m-2} \nonumber\\
\ll &\sum_{\substack{p \leq X \\ (p,2)=1}}\log p\bigg|\sum_{n\leq Y} \chi^{(8p)}(n)\bigg|^{4m-2}+\sum_{\substack{p \leq X \\ (p,2)=1}}\log p\bigg|\sum_{n\leq Y} \chi^{(8p)}(n)\Phi_U \Big( \frac {n}{Y} \Big)\bigg|^{4m-2}.
\end{align}
Since $4m-2=2(2m-1)\geq 2$ for $m\geq 1$, we can use Proposition \ref{Sm1} with $k=1$. Therefore, the second term on the right-hand side of (\ref{pocs3}) is
\begin{align}\label{pocs4}
 S^{(2)}_{2m-1}(X,Y)
 \ll  XY^{2m-1}(\log X)^{O_{m,\varepsilon}(1)}.
\end{align}
 To estimate the first term on the right-hand side of \eqref{pocs3}, we apply Lemma \ref{perron} with $w_n=\chi^{(8p)}(n)$, then
\begin{align}\label{Pint}
    &\sum_{n\leq Y}\chi^{(8p)}(n)= \frac 1{ 2\pi i}\int_{1+1/\log Y-iY}^{1+1/\log Y+iY}L(s,\chi^{(8p)}) \frac{Y^s}{s} d s +O(\log Y)\nonumber\\
     = & \left(\frac 1{ 2\pi i}\int_{1/2-iY}^{1/2+iY} +\frac 1{ 2\pi i}\int_{1+1/\log Y -iY}^{1/2-iY}+\frac 1{ 2\pi i}\int_{1/2+iY}^{1+1/\log Y+iY} \right)L(s,\chi^{(8p)})\frac{Y^s}{s}d s +O(\log Y),
\end{align}
 Inserting (\ref{Pint}) back into the first term of the right side of \eqref{pocs3}, we get
 \begin{align}\label{pocs5}
     \sum_{\substack{p \leq X \\ (p,2)=1}}\log p\bigg|\sum_{n\leq Y} \chi^{(8p)}(n)\bigg|^{4m-2}
     \ll (S_1+S_2+S_3+X(\log X)^{4m-2})\log X,
 \end{align}
where
\begin{align*}
    S_1=\sum_{\substack{p\leq X \\ (p,2)=1}}\bigg| \int_{1/2-iY}^{1/2+iY} L(s,\chi^{(8p)})\frac{Y^s}{s} ds\bigg|^{4m-2},\quad
    S_2=\sum_{\substack{p \leq X \\ (p,2)=1}}\bigg| \int_{1+1/\log Y-iY}^{1/2-iY} L(s,\chi^{(8p)})\frac{Y^s}{s} ds\bigg|^{4m-2}
\end{align*}
and
\begin{align*}
    S_3=\sum_{\substack{p \leq X \\ (p,2)=1}}\bigg| \int_{1/2+iY}^{1+1/\log Y+iY} L(s,\chi^{(8p)})\frac{Y^s}{s} ds\bigg|^{4m-2}.
\end{align*}
The estimates of $S_2$ and $S_3$ are the same, we just estimate $S_3$.
We assume that $Y\geq 10$, otherwise the lemma is trivial.
Let $m \geq 1$, which allows us to apply H?lder's inequality to get
\begin{align}\label{S3}
 S_3\ll & \sum_{\substack{p \leq X \\ (p,2)=1}}\bigg( \int_{1/2+iY}^{1+1/\log Y+iY} |L(s,\chi^{(8p)})| |ds| \bigg)^{4m-2} \nonumber\\
 \ll &  \sum_{\substack{p \leq X \\ (p,2)=1}}\int_{1/2+iY}^{1+1/\log Y+iY} |L(s,\chi^{(8p)})|^{4m-2} |ds| \ll X(\log X)^{8m^2-6m+1}.
\end{align}
For the last bound of (\ref{S3}), we utilize Lemma \ref{rude} with $ 1/2 \leq \sigma \leq 1+1/\log Y$ under GRH.
  Next we bound $S_1$ using H?lder's inequality, Lemma \ref{fdiff} with $k=1$ and the assumption $Y \leq X$.  Thus
\begin{align}\label{S1}
  S_1
  \ll & Y^{2m-1}\sum_{\substack{p \leq X \\ (p,2)=1}} \bigg( \int_{0}^Y \frac{|L( \tfrac{1}{2}+it,\chi^{(8p)}) |}{t+1} d t \bigg)^{4m-2}  \nonumber\\
  \ll & Y^{2m-1}\sum_{\substack{p \leq X \\ (p,2)=1}} \bigg( \big(\sum_{n\leq \log Y+2}1\big)^{\frac{4m-3}{4m-2}}\bigg(\sum_{n\leq \log Y+2}\bigg(\int_{e^{n-1}-1}^{e^{n}-1 } \frac{|L( \tfrac{1}{2}+it,\chi^{(8p)}) |}{t+1} d t\bigg)^{4m-2}\bigg)^{\frac{1}{4m-2}}\bigg)^{4m-2}\nonumber\\
   \ll &  Y^{2m-1}\sum_{n\leq \log Y+2} \frac{n^{4m-2} }{e^{(4m-2)n}} \sum_{\substack{p \leq X \\ (p,2)=1}} \bigg( \int_{e^{n-1}-1}^{e^{n}-1 } |L( \tfrac{1}{2}+it,\chi^{(8p)}) | d t \bigg)^{4m-2} \nonumber\\
  \ll & Y^{2m-1}X(\log X)^{O_{m,\varepsilon}(1)}\Big ( \sum_{n\leq \log Y+2}\frac{n^{4m-2}}{e^{(4m-2)n} }e^n +\sum_{n\leq \log Y+2}n^{4m-2} \Big )\nonumber\\
  \ll & Y^{2m-1}X(\log X)^{O_{m,\varepsilon}(1)}.
\end{align}
From \eqref{pocs5}--\eqref{S1}, we deduce that
\begin{equation}
\label{pocs6}
   \sum_{\substack{p \leq X \\ (p,2)=1}}\log p\bigg|\sum_{n\leq Y} \chi^{(8p)}(n)\bigg|^{4m-2}\ll XY^{2m-1}(\log X)^{O_{m,\varepsilon}(1)}.
\end{equation}
Finally, by \eqref{pocs1}--\eqref{pocs4}, \eqref{pocs6} and $ Y \leq X$, we complete the proof of the Proposition \ref{Sm2}.

\vspace*{.5cm}

\noindent{\bf Acknowledgments.} The author is grateful to Professor Peng Gao for useful discussions and suggestions. During the work the author is supported by NSFC grant 11871082.

\end{document}